\documentclass[a4paper, 12pt]{amsart}

\usepackage{graphicx} 
\usepackage{amssymb, amsmath, amsthm}
\usepackage{amscd}

\theoremstyle{plain}
\newtheorem{theorem}{Theorem}[section]
\newtheorem{lemma}[theorem]{Lemma}
\newtheorem{corollary}[theorem]{Corollary}
\newtheorem{definition}[theorem]{Definition}
\newtheorem{proposition}[theorem]{Proposition}

\newtheorem{remark}[theorem]{Remark}

\makeatletter

 \@addtoreset{equation}{section}
\makeatother

\newcommand{\sn}[0]{\mathrm{sn}}
\newcommand{\cs}[0]{\mathrm{cs}}

\newcommand{\diam}[0]{\mathrm{diam}}
\newcommand{\Lip}[0]{\mathrm{Lip}}
\newcommand{\mass}[0]{\mathrm{mass}}

\newcommand{\supp}[0]{\mathrm{supp}}
\newcommand{\cd}[0]{\mathrm{CD}}
\newcommand{\loc}[0]{\mathrm{loc}}
\newcommand{\mcp}[0]{\mathrm{MCP}}
\newcommand{\bg}[0]{\mathrm{BG}}

\title
[LLC condition for Alexandrov spaces]
{Locally Lipschitz contractibility of Alexandrov spaces and its applications}
\author{Ayato Mitsuishi and Takao Yamaguchi}
\address
{Mathematical Institute, Tohoku University, Sendai 980-8578, JAPAN}
\address
{Institute of Mathematics, University of Tsukuba, Tsukuba 305-8571, JAPAN}
\email[A.~Mitsuishi]{mitsuishi@math.tohoku.ac.jp}
\email[T.~Yamaguchi]{takao@math.tsukuba.ac.jp}
\date{\today}

\begin{document}
\maketitle

\begin{abstract}
We prove that any finite dimensional Alexandrov space with a lower curvature bound is locally Lipschitz contractible. 
As applications, we obtain a sufficient condition for solving the Plateau problem in an Alexandrov space considered by Mese and Zulkowski. 
%%%, and obtain an inequality between Gromov's simplicial volume and Hausdorff measure of an Alexandrov space. 
%%We estimate the simplicial volume of an Alexandrov space having a lower Ricci curvature bound in the sense of Bacher and Sturm.
\end{abstract}

\section{Introduction}
Alexandrov spaces are naturally appeared in the collapsing and convergence theory of Riemannian manifolds and played important roles in Riemannian geometry.
In the paper, when we say simply an Alexandrov space, it means that an Alexandrov space {\it of curvature bounded from below locally} and {\it of finite dimension}.
Their fundamental properties of such spaces were well studied in \cite{BGP}. 
There is a remarkable study of topological structures for Alexandrov spaces by Perelman \cite{Per Alex2}.
There, the topological stability theorem was proved which states that, if two compact Alexandrov spaces of the same dimension are very close in the Gromov-Hausdorff topology, then they are homeomorphic to each other. 
Further, it implies that for any point in an Alexandrov space, its small open ball is homeomorphic to its tangent cone. 
In particular, an open ball of small radius with respect to its center is contractible. 
It is expected by geometrers that the corresponding statements replacing homeomorphic by bi-Lipschitz homeomorphic could be proved. 
Until now, we did not know any Lipschitz structure of an Alexandrov space around singular points.
A main purpose of this paper is to prove that any finite dimensional Alexandrov space with a lower curvature bound is {\it strongly locally Lipschitz contractible} in the sense defined later.
For short, SLLC denotes this property.
The SLLC-condition is a strong version of the LLC-condition introduced in \cite{Y simp}
(cf.\, Remark \ref{rem:LLC}).
%%%The LLC-condition was introduced in \cite{Y simp}. 
%%%The second author showed some results by assuming the LLC-condition.

% The LLC-condition was introduced in \cite{Y simp}, where the simplicial volume and related quantities of metric spaces of curvature bounded form below or above were studied.
% He proved certain inequalities and equalities between those geometric and topological quantities by assuming the LLC-condition.
% By the definition, metric spaces of curvature bounded from above %%, which is recently almost called CAT-spaces, 
% satisfy the LLC-condition.  %%''«'¼'µ 
% However, it was not known whether Alexandrov spaces are LLC.

%%%% We recall the definition of the LLC-condition.
We define the strongly locally Lipschitz contractibility. 
%%For a point $p$ in a metric space $X$, we denote by $U(p,r)$ an open ball at $p$ of radius $r$.
We denote by $U(p,r)$ an open ball centered at $p$ of radius $r$ in a metric space. 

% $B(p,r)$ and $U(p,r)$ a closed ball and an open ball centered at $p$ of radius $r$, respectively: 
% \[
% B(p,r) = \{ q \in X \,|\, d(p, q) \le r \} \text{ and }
% U(p,r) = \{ q \in X \,|\, d(p, q) < r \}.
% \]
\begin{definition} \upshape \label{def of SLLC}
A metric space $X$ is {\it strongly locally Lipschitz contractible}, for short SLLC, if for every point $p \in X$, there exists $r > 0$ and a map 
\[
h : U(p, r) \times [0,1] \to U(p,r)
\]
such that $h$ is homotopy from $h(\cdot,0) = id_{U(p,r)}$ to $h(\cdot, 1) = p$, and it is Lipschitz, i.e., there exists $C, C' > 0$ such that 
\[
d(h(x,s), h(y,t)) \le C d(x,y) + C' |s -t|
\]
for every $x, y \in U(p,r)$ and $s, t \in [0,1]$, and for every $r' < r$, the image of $h$ restricted to $U(p,r') \times [0,1]$ is $U(p, r')$. 
\end{definition}
We call such a ball $U(p,r)$ a Lipschitz contractible ball and $h$ a Lipschitz contraction on $U(p,r)$.

%%%% The original definition is weaker than the LLC in Definition \ref{def of LLC}.
%%%% Alexandrov spaces of curvature bounded above (i.e.\! CAT-spaces) satisfy the LLC-condition.

%%%Since any Alexandrov space is locally cone-like space (\cite{Per Morse}, \cite{Per Alex2}), it is locally (topologically) contractible.
%%It was not known whether Alexandrov spaces are LLC.
%%It was just known that Alexandrov spaces are locally cone-like spaces via Perelman's Morse theory (and Stability theorem). 
A main result in the present paper is the following. 

\begin{theorem} \label{Alex sp is SLLC}
%%Let $X$ be a finite dimensional Alexandrov space and $p \in X$. 
%%Then, there exist a positive number $r$ and a Lipschitz map $\Phi : U(p, r) \times [0,1] \to U(p,r)$ such that, for any $x \in U(p,r)$, $\Phi(x,0) = x$ and $\Phi(x,1) = p$.
%%Further, the restriction $\Phi$ on $U(p, r') \times [0,1]$ for $r' < r$ has the image $U(p, r')$.
%%In particular, $X$ is LLC.
Any finite dimensional Alexandrov space is strongly locally Lipschitz contractible.
\end{theorem}

In \cite{Y simp}, a weaker form of Theorem \ref{Alex sp is SLLC} was conjectured. 
%%Note that this definition is a stronger form than the original definition of LLC in \cite{Y simp}. 

For metric spaces $P$ and $X$ and possibly empty subsets $Q \subset P$ and $A \subset X$, 
we denote by $f : (P,Q) \to (X, A)$ a map from $P$ to $X$ with $f(Q) \subset A$.
Two maps $f$ and $g$ from $(P,Q)$ to $(X,A)$ are {\it homotopic} (resp. {\it Lipschitz homotopic}) to each other if there exists a continuous (resp. Lipschitz) map 
\[
h : (P \times [0,1], Q \times [0,1]) \to (X,A)
\]
such that $h(x, 0) = f(x)$ and $h(x,1) = g(x)$ for all $x \in P$.
Then, we write $f \sim g$ (resp. $f \sim_{Lip} g$).
Let us denote by 
\[
[(P,Q), (X,A)] \text{ and } [(P,Q), (X,A)]_{Lip}
\]
the set of all homotopy classes of continuous maps from $(P,Q)$ to $(X,A)$ and the set of all Lipschitz homotopy classes of Lipschitz maps from $(P,Q)$ to $(X,A)$, respectively.

Let us consider a {\it Lipschitz simplicial complex} which means that it is a metric space and admits a triangulation such that each simplex is a bi-Lipschitz image of a simplex in a Euclidean space.
For precise definition, see Section \ref{proof of applications}.
\begin{corollary} \label{Lipschitz homotopy class is homotopy class}
Let $P$ be a finite Lipschitz simplicial complex and $Q$ a possibly empty subcomplex of $P$. 
Let $X$ be an Alexandrov space and $A$ an open subset of $X$.
Then, a natural map from $[(P,Q), (X,A)]_{Lip}$ to $[(P,Q), (X,A)]$ is bijective.
\end{corollary}

For %%an Alexandrov space 
a metric space $X$ and a point $x_0 \in X$ and $k \in \mathbb N$, we define {\it the $k$-th Lipschitz homotopy group} $\pi_k^{Lip}(X,x_0)$ by $\pi_k^{Lip}(X,x_0) = [(S^k, \ast), (X,x_0)]_{Lip}$ as sets, where $\ast \in S^k$ is an arbitrary point, equipped with a group operation as in the usual homotopy groups.
\begin{corollary} \label{Lipschitz homotopy group is homotopy group}
For an Alexandrov space $X$ and a point $x_0 \in X$ and $k \in \mathbb N$, a natural map
\[
\pi_k^{Lip} (X, x_0) \to \pi_k (X, x_0).
\]
is an isomorphism as groups.
\end{corollary}

\subsection{Application: the Plateau problem}

By Mese and Zulkowski \cite{MZ}, the Plateau problem in an Alexandrov space was considered as follows. 
Let $W^{1,2}(D^2, X)$ denote the $(1,2)$-Sobolev space from $D^2$ to an Alexandrov space $X$ in the sense of the Sobolev space of a metric space target defined by Korevaar and Schoen \cite{KS}.
Giving a closed Jordan curve $\Gamma$ in $X$, we set 
\begin{align*}
\mathcal F_\Gamma &:= \{u \in W^{1,2}(D^2,X) \cap C(D^2, X) \,;\, \\
& \hspace{60pt} u |_{\partial D^2} \text{ parametrizes } \Gamma \text{ monotonically}\}.
\end{align*}
They defined the area $A(u)$ of a Sobolev map $u \in W^{1,2}(D^2, X)$.
Under these settings, the Plateau problem is stated as follows. 

\vspace{0.5em}

\noindent
{\bf The Plateau problem}. 
Find a map $u \in W^{1,2}(D^2, X)$ such that 
\[
A(u) = \inf \{A(v) \,|\, v \in \mathcal F_\Gamma \}.
\]

They obtained 
\begin{theorem}[\cite{MZ}] \label{theroem by MZ}
Let $X$ be a finite dimensional compact Alexandrov space and $\Gamma$ be a closed Jordan curve in $X$.
If $\mathcal F_\Gamma \neq \emptyset$, then there exists a solution of the Plateau problem.
\end{theorem}

For an Alexandrov space, any condition of $\Gamma$ for implying $\mathcal F_\Gamma \neq \emptyset$ was not known.
As an application of Theorem \ref{Alex sp is SLLC}, we can obtain such a condition of $\Gamma$. 

\begin{corollary} \label{corollary: Plateau problem}
Let $\Gamma$ be a rectifiable closed Jordan curve in an Alexandrov space $X$. 
If $\Gamma$ is topologically contractible in $X$, then $\mathcal F_\Gamma \neq \emptyset$.
\end{corollary}
%%%%%%%%%..........

\subsection{Application: simplicial volume}
In \cite[Theorem 0.5]{Y simp}, the second author proved, assuming an LLC-condition on an Alexandrov space, an inequality between the Gromov's simplicial volume and the Hausdorff measure of it.
As an immediate consequence of Theorem \ref{Alex sp is SLLC}, we obtain 
%%%From it and Theorem \ref{Alex sp is SLLC}, we obtain 

\begin{corollary}[{cf.\! \cite{G}, \cite{Y simp}}] \label{cor:simp vol}
Let $X$ be a compact orientable $n$-dimensional Alexandrov space without boundary of curvature $\ge \kappa$ for $\kappa < 0$. 
Then, $\|X\| \le n!\,(n-1)^n \sqrt{-\kappa}^{\,n} \, \mathcal H^n(X)$.
\end{corollary}
%%%For precise terminology, we refer Subsection \ref{simplicial volume}.
Here, $\|X\|$ is the Gromov's simplicial volume which is the $\ell_1$-norm of the fundamental class of $X$, and $\mathcal H^n$ denotes the $n$-dimensional Hausdorff measure.
For precise terminologies, we refer \cite{G} and \cite{Y simp}.

Further, if we assume ``a lower Ricci curvature bound'' for $X$ in the sense of Bacher and Sturm \cite{BS}, then we obtain the following.

\begin{theorem} \label{thm:simp vol}
Let $X$ be a compact orientable $n$-dimensional Alexandrov space without boundary. 
Let $m$ be a locally finite Borel measure on $X$ with full support which is absolutely continuous with respect to $\mathcal H^n$.
If the metric measure space $(X,m)$ satisfies the reduced curvature-dimension condition $\cd^\ast(K,N)$ locally for $K, N \in \mathbb R$ with $N \ge 1$ and $K < 0$, then 
\[
\|X\| \le n!\, \sqrt{-(N-1)K}^{\, n} \mathcal H^n(X).
\]
\end{theorem}

Theorem \ref{thm:simp vol} is new even if $X$ is a manifold, because a reference measure $m$ can be freely chosen. 

\vspace{.5em}
\noindent
{\bf Organization}.
We review fundamental properties of Alexandrov spaces in Section \ref{preliminaries}. 
In particular, we recall the theory of the gradient flow of distance functions on an Alexandrov space established by Perelman and Petrunin \cite{PP QG}.
In Section \ref{proof of main theorem}, we prove that the distance function from a metric sphere at each point in an Alexandrov space is regular on a much smaller concentric punctured ball. 
Then, using the flow of it, we prove Theorem \ref{Alex sp is SLLC}.
In Section \ref{proof of applications}, we recall precise terminologies in the applications in the introduction, and prove Corollaries \ref{Lipschitz homotopy class is homotopy class}, \ref{Lipschitz homotopy group is homotopy group} and \ref{cor:simp vol}. 
In Section \ref{sec:infinite dim}, we note that our proof given in Section \ref{proof of main theorem} also works for infinite dimensional Alexandrov spaces whenever the space of directions is compact.
In Section \ref{sec:simp vol}, we recall several notions of a lower Ricci curvature bound on metric space togerther with a Borel measure and their relation. 
%%%We explain that the Bishop-Gromov type volume growth inequality implied by a generalized lower Ricci curvature bound can be used to prove Theorem \ref{thm:simp vol}.
By using the Bishop-Gromov type volume growth inequality, we prove Theorem \ref{thm:simp vol}.

\section{Preliminaries} \label{preliminaries}
This section consists of just a review of the definition of Alexandrov spaces and somewhat detailed review of the gradient flow theory of semiconcave functions on Alexandrov spaces.
For precisely, we refer \cite{BGP}, \cite{BBI} or \cite{Pet semi}.

We recall the definition of Alexandrov spaces. 
\begin{definition}[cf. \cite{BBI}, \cite{BGP}] \upshape \label{def:Alexandrov space}
Let $\kappa \in \mathbb R$.
We call a complete metric space $X$ an {\it Alexandrov space of curvature $\ge \kappa$} if it satisfies the following. 
\begin{itemize}
\item[(1)] $X$ is a geodesic space, i.e., for every $x$ and $y$ in $X$, there is a curve $\gamma : [0, |x,y|] \to X$ such that $\gamma(0) = x$ and $\gamma(|x,y|) = y$ and the length $L(\gamma)$ of $\gamma$ equals $|x,y|$.
Here, $|x,y|$ denotes the distance between $x$ and $y$.
We call such a curve $\gamma$ a geodesic between $x$ and $y$, and denote it by $x y$.
\item[(2)] $X$ has curvature $\ge \kappa$, i.e., for every $p$, $q$, $r \in X$ $($\!\! with $|p, q| + |q, r| + |r, p| < 2 \pi / \sqrt \kappa$ if $\kappa > 0$\!\! $)$ and every $x$ in a geodesic $q r$ between $q$ and $r$, taking a comparison triangle $\triangle \tilde p \tilde q \tilde r = \tilde \triangle p q r$ in a simply-connected complete surface $\mathbb M_\kappa$ of constant curvature $\kappa$, and corresponding point $\tilde x$ in $\tilde q \tilde r$, we have 
\[
|p, x| \ge |\tilde p, \tilde x|.
\]
\end{itemize}
We simply say that a complete metric space $X$ is an {\it Alexandrov space} if it is a geodesic space and for any $p \in X$, there exists a neighborhood $U$ of $p$ and $\kappa \in \mathbb R$ such that $U$ has curvature $\ge \kappa$ as the condition $(2)$, i.e., any triangle in $U$ (whose sides contained in $U$) is not thinner than its comparison triangle in $\mathbb M_\kappa$.
\end{definition}
If $X$ is compact, then it has a uniform lower curvature bound.
Throughout in the paper, we do not need a uniform lower curvature bound, since we are mainly interested in a local property.
It is known that if $X$ has a uniform lower curvature bound, say $\kappa$, then $X$ has curvature $\ge \kappa$ (\cite{BGP}).

\subsection{Semiconcave functions} \label{subsection: semiconcave function}
%%%% \label{sec:gradient}
%%%% In this section, we recall the notion of concave function, and the integral flow of the gradient of such a function on an Alexandrov space.
%%%% We always assume that any geodesic is arclength parametrized.

%%%% First, we recall the definition of a concave function on an interval.

In this subsection, we refer \cite{Pet semi} and \cite{Pet QG}.
\begin{definition} \upshape \label{concave function on interval}
Let $I$ be an interval and $\lambda \in \mathbb R$.
We say a function $f : I \to \mathbb R$ to be {\it $\lambda$-concave} if the function 
\[
\bar f(t) = f(t) - \frac{\lambda}{2} t^2
\]
is concave on $I$.
Namely, for any $t < t' < t''$ in $I$, we have
\[
\frac{\bar f(t')-\bar f(t)}{t' -t} \ge \frac{\bar f(t'') - \bar f(t')}{t''-t'}.
\]

We say a function $f : I \to \mathbb R$ to be {\it $\lambda$-concave in the barrier sense} if for any $t_0 \in \mathrm{int}\, I$, there exist a neighborhood $I_0$ of $t_0$ in $I$ and a twice differentiable function $g : I_0 \to \mathbb R$ such that 
\begin{equation*}
g(t_0) = f(t_0),\, g \ge f \text{ and } g'' \le \lambda \text{ on } \mathrm{int}\, I.
\end{equation*}
\end{definition}

\begin{lemma}[cf.\! \cite{Pet QG}] \label{concavity are same}
Let $f : I \to \mathbb R$ be a continuous function on an interval $I$ and $\lambda \in \mathbb R$.
Then the following are equivalent.
\begin{itemize}
\item[(1)] $f$ is $\lambda$-concave in the sense of Definition \ref{concave function on interval};
\item[(2)] For any $t_0 \in I$, there is $A \in \mathbb R$ such that 
\[
f(t) \le f(t_0) + A (t- t_0) + \frac{\lambda}{2} (t-t_0)^2
\]
for any $t \in I$.
\item[(3)] $f$ is $\lambda$-concave in the barrier sense.
\end{itemize}
\end{lemma}
\begin{proof}
By considering $f(t) - (\lambda /2)\, t^2$, we may assume that $\lambda = 0$.

Let us prove the implication $(1) \Rightarrow (2)$.
Let us take $t_0 \in I$ to be not the maximum number of $I$.
By the concavity of $f$, the value 
\[
A = \lim_{\varepsilon \to 0+} \frac{f(t_0 + \varepsilon)-f(t_0)}{\varepsilon}
\]
is well-defined.
And, the concavity of $f$ implies 
\[
f(t) \le f(t_0) + A (t-t_0).
\]
When $t_0 \in I$ is the maximum value of $I$, then replacing $A$ to the limit $\lim_{\varepsilon \to 0+}(f(t_0-\varepsilon)-f(t_0)) / \varepsilon$, we obtain the inequality as same as above.

The implication $(2) \Rightarrow (3)$ is trivial.

Let us assume that $f$ satisfies $(3)$. Let us take $t_0$ in the interior of $I$. Then there exists a twice differentiable function $g : I \to \mathbb R$ such that 
\[
g(t_0) = f(t_0),\, g \ge f \text{ and } g'' \le 0.
\]
Hence, for any $t' < t_0 < t$, we have 
\[
\frac{f(t)-f(t_0)}{t-t_0} 
\le \frac{g(t)-g(t_0)}{t-t_0} 
\le \frac{g(t_0)-g(t')}{t_0-t'} 
\le \frac{f(t_0)-f(t')}{t_0-t'}.
\]
Therefore, $f$ is concave.
\end{proof}
Let $X$ be a geodesic space and $U$ be an open subset of $X$.
Let $f : U \to \mathbb R$ be a function.
We say that $f$ is {\it $\lambda$-concave} on $U$ if for every geodesic $\gamma : I \to U$, the function $f \circ \gamma : I \to \mathbb R$ is $\lambda$-concave on $I$.
For a function $g : U \to \mathbb R$, we say that $f$ to be {\it $g$-concave} if for any $p \in U$ and $\varepsilon > 0$, there is an open neighborhood $V$ of $p$ in $U$, such that $f$ is $(g(p)+\varepsilon)$-concave on $V$.
%%%% For a function $g : U \to \mathbb R$, we say that $f : U \to \mathbb R$ is {\it $g$-concave} on $U$ if for any $p \in U$ and $\varepsilon > 0$, there exists an open neighborhood $V$ of $p$ in $U$ such that $f|_V$ is $(g(p)+\varepsilon)$-concave.
We say that $f : U \to \mathbb R$ is {\it $g$-concave in the barrier sense} if for any $p \in U$ and $\varepsilon > 0$, there exists an open neighborhood $V$ of $p$ in $U$ such that %%$f|_V$ is $(g(p) + \varepsilon)$-concave in the barrier sense.
for every geodesic $\gamma$ contained in $V$, $f \circ \gamma$ is $(g(p)+\varepsilon)$-concave in the barrier sense.
By a similar argument to the proof of Lemma \ref{concavity are same}, $f$ is $g$-concave if and only if $f$ is $g$-concave in the barrier sense.

From now on, we fix an Alexandrov space $X$. %%%%with curvature bounded below by $\kappa$.
We use results and notions on Alexandrov spaces obtained in \cite{BGP}, and we refer \cite{BBI}.
$T_p X$ denotes the tangent cone of $X$ at $p$, and $\Sigma_p X$ denotes the space of directions of $X$ at $p$.

For any $\lambda$-concave function $f : U \to \mathbb R$ on an open subset $U$ of $X$ and $p \in U$, and $\delta > 0$,
a function $f_\delta : \delta^{-1} U \to \mathbb R$ is defined by the same function $f_\delta = f$ on the same domain $\delta^{-1}U = U$ as sets. 
Since the metric of $\delta^{-1} U$ is the metric of $U$ multiplied by $\delta^{-1}$, $f_\delta$ is $\delta^2 \lambda$-concave on $\delta^{-1} U$.
In addition, if $f$ is Lipschitz near $p$, then the blow-up $d_p f : T_p X \to \mathbb R$, namely the limit with respect to some sequence $\delta_i \to 0$, 
\[
\lim_{i \to \infty} f_{\delta_i} : \lim_{i \to \infty} (\delta_i^{-1}U,p) \to \mathbb R
\]
is $0$-concave on $T_p X$.
$d_p f$ is called the {\it differential} of $f$ at $p$.
%%%% The derivation $d_p f$ depends on sequence $(\delta_i)$, but we omit $(\delta_i)$.
Note that the differential of locally Lipschitz semiconcave function always exist and does not depend on the choice of a sequence $(\delta_i)$.
Actually, $d_p f (\xi)$ is calculated by 
\[
d_p f(\xi) = \lim_{t \to 0+} \frac{f(\exp_p (t \xi)) - f(p)}{t}
\]
if $\xi \in \Sigma_p'$ is a geodesic direction, where $\exp_p(t \xi)$ denotes the geodesic starting from $p$ with the direction $\xi$.

\subsection{Distance functions as semiconcave functions}
For any real number $\kappa$, let us define ``trigonometric functions'' $\sn_\kappa$ and $\cs_\kappa$ by the following ODE. %%%% \eqref{ODE:trigonometric}.
\begin{equation*} \label{ODE:trigonometric}
\left\{
\begin{aligned}
&\sn_\kappa''(t) + \kappa \sn_\kappa(t) = 0, & \sn_\kappa (0) = 0, && \sn_\kappa'(0) = 1;\\
&\cs_\kappa''(t) + \kappa \cs_\kappa(t) = 0, & \cs_\kappa (0) = 1, && \cs_\kappa'(0) = 0.
\end{aligned}
\right.
\end{equation*}
They are explicitly represented as follows.
%%%% Explicitly, they are as follows.
\[
\sn_\kappa(t) = \sum_{n=0}^\infty \frac{(-\kappa)^n}{(2 n+1)!}\, t^{2 n+1} = \left\{ 
\begin{aligned}
& \frac{1}{\sqrt \kappa} \sin(\sqrt \kappa\, t) & (\text{if } \kappa > 0)\\
& t & (\text{if } \kappa = 0)\\
& \frac{1}{\sqrt{-\kappa}} \sinh (\sqrt{-\kappa}\, t) & (\text{if } \kappa < 0)
\end{aligned}
\right.
\]
\[
\cs_\kappa(t) = \sn_\kappa'(t) =\sum_{n=0}^\infty \frac{(-\kappa)^n}{(2 n)!}\, t^{2 n} = \left\{ 
\begin{aligned}
& \cos(\sqrt \kappa\, t) & (\text{if } \kappa > 0)\\
& 1 & (\text{if } \kappa = 0) \\
& \cosh (\sqrt{-\kappa}\, t) & (\text{if } \kappa < 0)
\end{aligned}
\right.
\]
These functions are elementary for the space form $\mathbb M_\kappa$
in the sense that they satisfy the following.
%%%% Namely these satisfy the following. 
Let us take any points $p, q, r \in \mathbb M_\kappa$ with $|pq|+|qr|+|rp| < 2\, \diam\, \mathbb M_\kappa$, and set $\theta := \angle qpr$.
Let $\gamma$ be the geodesic $pr$ with $\gamma(0) = p$ and $\gamma(|p,r|) = r$.
We set $\ell(t) = |q, \gamma(t)|$. 
When $\kappa \neq 0$, the cosine formula states
\begin{equation*} \label{cosine formula}
\cs_\kappa (\ell(t)) = \cs_\kappa |p q|\, \cs_\kappa t + \kappa\, \sn_\kappa |p q|\, \sn_\kappa t \cos \theta.
\end{equation*}
%%%Let us consider a triangle $\triangle p q r$ in $\mathbb M_\kappa$ with $|p q| + |q r|+ |r p| < 2\, \diam\, \mathbb M_\kappa$.
Then, we have 
\begin{equation} \label{cosine formula2}
\left(\cs_\kappa (\ell (t))\right)'' + \kappa\, \cs_\kappa (\ell(t)) = 0.
\end{equation}

\begin{lemma}[cf.\,\cite{PP QG}] \label{1-kf concave}
The distance function $d_A$ from a closed subset $A$ in an Alexandrov space $X$ of curvature $\ge \kappa$ is $(\cs_\kappa (d_A) / \sn_\kappa (d_A))$-concave on $(X - A) \cap \{d_A < \frac{\pi}{2 \sqrt \kappa} \}$. 
Here, if $\kappa \le 0$, then we consider $\frac{\pi}{2 \sqrt \kappa}$ as $+ \infty$, and if $\kappa = 0$, then we consider $\cs_\kappa (d_A) / \sn_\kappa (d_A)$ as $1 / d_A$.
\end{lemma}

\begin{proof}
We consider the case that $\kappa \neq 0$.
Let us take any geodesic $\gamma$ contained in $(X - A) \cap \{d_A < \frac{\pi}{2 \sqrt \kappa}\}$.
We take $x$ on $\gamma$ and reparametrize $\gamma$ as $x = \gamma(0)$.
We choose $w \in A$ such that $|A x| = |wx|$.
We set $\ell(t) := |A, \gamma(t)|$. 
Let us take a geodesic $\tilde \gamma$ and a point $\tilde w$ in the $\kappa$-plane $\mathbb M_\kappa$ such that $|\tilde w \tilde \gamma(0)| = |w x|$ and $\angle (\uparrow_{\tilde x}^{\tilde w}, \tilde \gamma^+(0)) = \angle (\uparrow_x^w, \gamma^+(0))$.
Let us set $\tilde \ell(t) := |\tilde w, \tilde \gamma(t)|$. %% and  $\tilde f (t) := \rho_\kappa (\tilde \ell(t))$.
By the Alexandrov convexity, $\ell(t) \le \tilde \ell(t)$. %% and $\ell(0) = \tilde \ell(0)$.
Noticing the sign, we obtain
\[
- \frac{1}{\kappa} \cs_\kappa (\ell) \le - \frac{1}{\kappa} \cs_\kappa (\tilde \ell).
\]
% Hence, we have 
% \[
% f(t) \le \tilde f(t) \text{ and } f(0) = \tilde f(0).
% \]
By the calculation \eqref{cosine formula2}, we have 
\[
\left( - \frac{1}{\kappa} \cs_\kappa (\ell) \right)'' \le \cs_\kappa (\ell)
\]
at $t = 0$ in the barrier sense.
On the other hand, we can calculate the second derivative as   
\[
\left( \cs_\kappa \circ \ell (t) \right)'' = - \kappa \left[ \cs_\kappa (\ell) \cdot (\ell')^2 + \sn_\kappa (\ell) \cdot \ell'' \right]
\]
in the barrier sense. 
Noticing that $\cs_\kappa (\ell) \ge 0$ if $\ell \le \frac{\pi}{2 \sqrt \kappa}$, we obtain 
\[
\sn_\kappa (\ell) \cdot \ell'' \le \cs_\kappa (\ell)
\]
at $t = 0$ in the barrier sense. 
It completes the proof of the lemma if $\kappa \neq 0$. 
When $X$ has nonnegative curvature, taking a negative number $\kappa$ as a lower curvature bound of $X$ and tending $\kappa$ to $0$, we obtain $\cs_\kappa (d_A) / \sn_\kappa (d_A) \to 1 / d_A$.
\end{proof}

\subsection{Gradient flows}
In this subsection, we refer \cite{Pet semi}, \cite{Pet QG} and \cite{PP QG}.

For vectors $v, w$ in the tangent cone $T_p X$, setting $o = o_p$ the origin of $T_pX$, we define $|v| = |o,v|$ and 
\[
\langle v, w \rangle = \left\{ 
\begin{aligned}
&|v| |w| \cos \angle v o w &&\text{ if } |v|, |w| > 0 \\
&0 &&\text{ otherwise. }
\end{aligned} 
\right.
\]
\begin{definition}[\cite{PP QG}, \cite{Pet QG}] \upshape \label{def:semiconcave}
Let $f$ be a $\lambda$-concave function on an open subset $U$ of $X$.
We say that a vector $g \in T_p X$ at $p \in U$ is a {\it gradient} of $f$ at $p$ if it satisfies 
\begin{align*}
(1)\,\, 
& d f_p(v) \le \langle v , g \rangle \text{ for all } v \in T_p X;
\\
(2)\,\,
& d f_p(g) = \langle g , g \rangle.
\end{align*}
We recall that $g$ uniquely exists. 

We say that $f$ is {\it regular at} $p$ if $d_p f(v) > 0$ for some $v \in T_p X$, equivalently, $|\nabla_p f| > 0$.
Otherwise, $f$ is said to be {\it critical at} $p$.
\end{definition}

\begin{definition}[\cite{PP QG}, \cite{Pet QG}] \upshape
Let $f : U \to \mathbb R$ be a semiconcave function on an open subset $U$ of an Alexandrov space. 
A Lipschitz curve $\gamma : [0,a) \to X$ on an interval $[0,a)$ is said to be a {\it gradient curve on $U$} if 
for any $t \in [0,a)$ with $\gamma(t) \in U$, 
\[
\lim_{\varepsilon \to 0+} \frac{f \circ \gamma(t + \varepsilon) - f \circ \gamma(t)}{\varepsilon}
\]
exists and it is equals to $|\nabla f|^2 (\gamma(t))$.
\end{definition}
Note that if $f$ is critical at $\gamma(t)$, the gradient curve $\gamma$ for $f$ satisfies that $\gamma(t') = \gamma(t)$ for any $t' \ge t$.

The (multi-valued) logarithm map $\log_p : X \to T_p X$ is defined by if $x \neq p$, then $\log_p (x) = |px|\, \cdot\! \uparrow_p^x$, where $\uparrow_p^x$ is a direction of a geodesic $px$ and if $x = p$, then $\log_p(x) = o_p$.
If $\gamma$ is a gradient curve on $U$, then for $t$ with $\gamma(t) \in U$, the forward direction 
\[
\gamma^+(t) := \lim_{\varepsilon \to 0+} \frac{\log_{\gamma(t)} (\gamma(t + \varepsilon))}{\varepsilon} \in T_{\gamma(t)} X
\]
exists and it is equals to the gradient $\nabla f (\gamma(t))$.

\begin{proposition}[{\cite{KPT}, \cite{Pet semi}, \cite{Pet QG}, \cite{PP QG}}]
\label{gradient flow is Lipschitz}
Let $\gamma$ and $\eta$ be gradient curve starting from $x = \gamma(0)$ and $y = \eta(0)$ in an open subset $U$ for a $\lambda$-concave function $f : U \to \mathbb R$, we obtain 
\[
|\gamma(s) \eta(s)| \le e^{\lambda s} |x y|
\]
for every $s \ge 0$.
\end{proposition}

This proposition implies a gradient curve starting at $x \in U$ is unique on the domain of definition.

\begin{theorem}[{\cite{Pet QG}, \cite{Pet semi}, \cite{PP QG}}] \label{existence of gradient curve}
For any open subset $U$ of an Alexandrov space, a semiconcave function $f$ on $U$ and $x \in U$, there exists a unique maximal gradient curve
\[
\gamma : [0,a) \to U
\]
with $\gamma(0) = x$ for $f$, where $\gamma$ is maximal if for every gradient curve $\eta : [0, b) \to U$ for $f$ with $\eta(0) = x$, then $b \le a$.
\end{theorem}

\begin{definition}[\cite{PP QG}, \cite{Pet QG}] \upshape
Let $U$ be an open subset of an Alexandrov space $X$ and $f : U \to \mathbb R$ is semiconcave function.
Let $\{[0,a_x)\}_{x \in U}$ be a family of intervals for $a_x > 0$.
A map 
\[
\Phi : \bigcup_{x \in U} \{x\} \times [0,a_x) \to U
\]
is a {\it gradient flow} of $f$ on $U$ (with respect to $\{[0,a_x)\}_{x \in U}$) if 
for every $x \in U$, $\Phi(x,0) = x$ and the restriction 
\[
\Phi(x, \cdot) : [0,a_x) \to U
\]
is gradient curve of $f$ on $U$.

A gradient flow $\Phi$ is {\it maximal} if each domain $[0,a_x)$ of the gradient curve is maximal. 
\end{definition}
By Theorem \ref{existence of gradient curve} and Proposition \ref{gradient flow is Lipschitz}, a maximal gradient flow on $U$ always uniquely exists.

Let $\Phi$ be a gradient flow of a semiconcave function on an open subset $U$. 
By a standard argument, we obtain 
\[
\Phi(x, s + t) = \Phi(\Phi(x,s),t)
\]
for every $x \in U$ and $s, t \ge 0$, whenever the formula has the meaning.

\section{Proof of Theorem \ref{Alex sp is SLLC}} \label{proof of main theorem}

The purpose of this section is to prove Theorem \ref{Alex sp is SLLC}.
Let us fix a finite dimensional Alexandrov space $X$. %%%of curvature $\ge -1$.
As see before in Section \ref{sec:infinite dim}, this construction makes a sense for an infinite dimensional Alexandrov space with additional assumption.

We first prove the following.
Let us consider the distance function $f = d(S(p, R), \cdot)$ from a metric sphere $S(p, R) = \{q \in X \,\big|\, |pq| = R\}$.
We may assume that a neighborhood of $p$ has curvature $\ge -1$ by rescaling the metric of $X$ if necessary.
By $B(p,R)$, we denote the closed ball centered at $p$ of radius $R$.

\begin{proposition} \label{distance function from boundary}
For any $p \in X$ and $\varepsilon > 0$, there exists $R > 0$ and $\delta_0 = \delta_0(\varepsilon, R) > 0$ such that the distance function 
\[
f = d(S(p, R), \cdot)
\]
from the metric sphere $S(p, R)$ satisfies that for every $x \in B(p, \delta_0 R) - \{p\}$, 
\begin{equation} \label{important eq}
d_x f (\uparrow_x^p) > \cos \varepsilon.
\end{equation}
In particular, $f$ is regular on $B(p, \delta_0 R) - \{p\}$.
%%%% Here, $\tau(\delta)$ is a positive constant depending on $\delta$ satisfying $\lim_{\delta \to 0} \tau(\delta) = 0$.
\end{proposition}

\begin{remark}\upshape
In \cite{Per Morse}, Perelman constructed a strictly concave function on a small neighborhood of each point by taking an average of composition of a strictly concave polynomial and distance functions. 
Kapovitch also constructed such a function in \cite{K reg} and \cite{K rest} based on Perelman's construction.
We explain their construction. 
For $\delta >0$, let $\phi$ be a second degree polynomial such that $\phi'' \equiv -1/\delta$, $\phi'$ is positive and bounded on small interval $[R-4 \delta, R+4 \delta]$.
Taking a sufficiently fine net $\{x_\alpha\}_{\alpha = 1}^N$ in $S(p, R)$ and much finner net $\{x_{\alpha \beta}\}_{\beta = 1}^{N_\alpha}$ near each $x_\alpha$, %%in $B(x_\alpha,\delta R) \cap S(p,R)$
we define 
\[
f_\alpha (x) = \frac{1}{N_\alpha} \sum_{\beta = 1}^{N_\alpha} \phi (d(x_{\alpha \beta},x)).
\]
They proved that $f_\alpha$ is $(-c/\delta)$-concave on $B(p,\delta)$ in the sense of Definition \ref{def:semiconcave}, where $c$ is a positive constant independent on $\delta$.
Here, we note that a $\lambda$-concave function is $(-\lambda)$-concave in their term.
Then, $f = \min_\alpha f_\alpha$ is also $(-c/\delta)$-concave.
In particular, it follows that any point has a convex neighborhood.
%%This function played an important role in \cite{K reg} and \cite{K rest}.
One can prove that $f$ has a unique maximum at $p$ in $B(p, \delta)$.
Immediately, by Proposition \ref{gradient flow is Lipschitz}, the gradient flow $\Phi$ of $f$ satisfies $d(\Phi(x,t), p) \le e^{-t c/\delta} d(x,p)$ for any $x \in B(p,\delta)$.
Hence, the flow curve $\Phi(x,t)$ is tending to $p$ as $t \to \infty$.
However, we do not know whether $\Phi(x,\cdot)$ reach $p$ at finite time. %%%for some $x \in B(p,\delta)$ and 
In particular, we do not know whether $d f(\uparrow_x^p)$ has a uniform positive lower bound in $B(p,\delta) - \{p\}$ as the conclusion of Proposition \ref{distance function from boundary}.
Proposition \ref{distance function from boundary} is a key in our paper, which implies an important Lemma \ref{upper bound of time} later.
\end{remark}

\begin{proof}[Proof of Proposition \ref{distance function from boundary}]
Since the tangent cone $T_p X$ is isometric to the metric cone $K(\Sigma_p)$ over the space of directions $\Sigma_p$, there exists a positive constant $R$ satisfying the following.
\begin{equation} \label{eq 0}
\text{For any } v \in \Sigma_p,\! \text{ there is } q \in S(p, R) \text{ such that } \angle(v, \uparrow_p^q) \le \varepsilon.
\end{equation}
From now on, we set $S := S(p, R)$.
For any $x \in S(p, \delta R)$, fixing a direction $\uparrow_p^x\,\, \in x'_p$, let us take $q_1, q_2 \in S$ such that 
\begin{align}
|x, q_1| &= |x, S| := \min_{q \in S} |x, q|; \label{eq q_1} \\
\angle x p q_2 &= \angle (\uparrow_p^x, \uparrow_p^{q_2}) = \angle (\uparrow_p^x, S_p') := \min_{v \in S_p'} \angle(\uparrow_p^x, v). \label{eq q_2}
\end{align}
By the condition \eqref{eq 0}, we have 
\[
\tilde \angle x p q_2 \le \angle x p q_2 \le \varepsilon.
\]
Then, by the law of sine, we obtain 
\begin{equation} \label{eq 1}
\begin{aligned}
\sin \tilde \angle p x q_2 &= \frac{\sinh R}{\sinh |x q_2|} \sin \tilde \angle x p q_2 
%%%% \\ &
\le \frac{\sinh R}{\sinh R (1-\delta)} \sin \varepsilon. 
%%%% \le \frac{1+R}{1-\delta} \sin \varepsilon.
\end{aligned}
\end{equation}
On the other hand, by the law of cosine,
%%%% and the Alexandrov convexity, 
we obtain
\begin{align*}
\cosh |x q_2| &= \cosh \delta R \cosh R - \sinh \delta R \sinh R \cos \tilde \angle x p q_2 \\
& \le \cosh \delta R \cosh R - \sinh \delta R \sinh R \cos \varepsilon 
%%%% \\
%%%% & = \cosh (R - \delta R) + \sinh R \sinh \delta R (1 - \cos \varepsilon)
\end{align*}
and 
\begin{align*}
& - \sinh \delta R \sinh |x q_2| \cos \tilde \angle p x q_2 
 = \cosh R - \cosh \delta R \cosh |x q_2| \\
%%%% & \hspace{12pt} \ge \cosh R - \cosh \delta R \{ \cosh (R - \delta R) + \sinh R \sinh \delta R (1 - \cos \varepsilon) \}.
& \hspace{16pt} \ge \cosh R \{1 - \cosh^2 \delta R \} + \sinh R \sinh \delta R \cos \varepsilon.
\end{align*}
Therefore, if $\delta$ is smaller than some constant, 
%%%% with respect to $R$, 
%%%% and $\varepsilon$, 
then 
\begin{equation} \label{eq 2}
- \cos \tilde \angle p x q_2 > 0.
\end{equation}
By \eqref{eq 1} and \eqref{eq 2}, we obtain 
\begin{equation} \label{p x q_2}
\tilde \angle p x q_2 \ge \pi - (1+\tau(\delta)) \varepsilon.
\end{equation}

Next, let us consider the point $q_1$ taken as \eqref{eq q_1}.
Then, it satisfies 
\[
\tilde \angle x p q_1 = \min_{q \in S} \tilde \angle x p q \le \min_{q \in S} \angle x p q \le \varepsilon.
\]
By a similar argument as $q_1$ instead of $q_2$, we obtain 
\begin{equation} \label{p x q_1}
\tilde \angle p x q_1 \ge \pi - (1+\tau(\delta)) \varepsilon.
\end{equation}

By the quadruple condition with \eqref{p x q_2} and \eqref{p x q_1}, we obtain 
\[
%%%\tilde \angle q_1 x q_2 \le \angle q_1 x q_2 \le 2 \pi - \angle p x q_1 - \angle p x q_2 \le (2 + \tau(\delta)) \varepsilon.
\tilde \angle q_1 x q_2 \le 2 \pi - \tilde \angle p x q_1 - \tilde \angle p x q_2 \le (2 + \tau(\delta)) \varepsilon.
\]
If $\delta$ is small with respect to $\varepsilon$, then we obtain 
\[
|q_1 q_2| \le 3 R \varepsilon.
\]
Therefore, we obtain 
\begin{equation} \label{q_1 p q_2}
\tilde \angle q_1 p q_2 \le 4 \varepsilon.
\end{equation}

%%We prove that $f = d(S, \cdot)$ is $(1 - C \varepsilon)$-regular on $B(p, \delta R) - \{p\}$. Here, $C$ is some constant depending on $\delta$ and $p$. %%%% check ?????
For any $y \in p x - \{p, x \}$, we set $q_3 = q_3(y) \in S$ such that 
\[
|y, q_3| = |y, S|. %%%% |x, q_3| = |x, S|.
\]
By an argument as above, we obtain 
\begin{equation} \label{p y q_3}
\tilde \angle p y q_3 \ge \pi - (1 + \tau (|p y| / R)) \varepsilon > \pi - 2 \varepsilon.
\end{equation}
Then, we have 
\[
\tilde \angle x y q_3 < 2 \varepsilon.
\]
By the Gauss-Bonnet's theorem, 
if $y$ is near $x$, then 
\[
\tilde \angle y x q_3 > \pi - 3 \varepsilon.
\]
% By Lemma \ref{first variation formula} later, we have
% \begin{align*}
% &\frac{\cosh |S y| - \cosh |S x|}{\sinh |x y|} \ge 
% \frac{\cosh |q_3 y| - \cosh |q_3 x|}{\sinh |x y|} \\
% &\ge 
% - \sinh |q_3 x| \cos \tilde \angle q_3 x y - |x y| \cosh |q_3 x| \\
% &\ge \sinh |S x| \cos (3 \varepsilon) - |x y| \cosh |q_3 x|
% \end{align*}
% Tending $y \to x$ along the geodesic $x p$, we obtain 
% \[
% \sinh |S x| \lim_{x p \ni y \to x} \frac{|S y| - |S x|}{|x y|} \ge \sinh |S x| \cos (3 \varepsilon).
% \]
% It follows 
% \[
% d f_x (\uparrow_x^p) = \lim_{x p \ni y \to x} \frac{|S y| - |S x|}{|x y|} \ge \cos 3 \varepsilon.
% \]
% This completes the proof.
By the first variation formula, we obtain 
\[
d f_x (\uparrow_x^p) = \lim_{x p \ni y \to x} \frac{|S y| - |S x|}{|x y|} \ge 
\liminf_{x p \ni y \to x} \frac{|q_3 y| - |q_3 x|}{|x y|} \ge \cos 3 \varepsilon.
\]
This completes the proof.
\end{proof}

% \begin{lemma} \label{first variation formula}
% For a triangle $\triangle q x y$ in the hyperbolic plane $\mathbb H^2$, setting $\theta = \angle q x y$, we obtain 
% \[
% \left| \frac{\cosh |q y| - \cosh |q x|}{\sinh |x y|} + \sinh |q x| \cos \theta \right| \le |x y| \cosh |q x|.
% \]
% \end{lemma}
% \begin{proof}
% From the cosine formula, we have 
% \begin{align*}
% &\left| \frac{\cosh |q y| - \cosh |q x|}{\sinh |x y|} + \sinh |q x| \cos \theta \right| 
% \\ 
% &\le (\cosh |q x| - 1) \frac{\cosh |x y| -1}{\sinh |x y|} + \frac{\cosh |x y| -1}{\sinh |x y|}.
% \end{align*}
% Since $\frac{\cosh t -1}{\sinh t} \le t$ for every $t \ge 0$, we obtain the conclusion.
% \end{proof}

%%%% Let us start the proof of Theorem \ref{Alex sp is SLLC}.

We fix $\delta_0$ as in the conclusion of Proposition \ref{distance function from boundary} and fix $\delta \le \delta_0$.
\begin{lemma}
For any $x \in B(p, \delta R) - \{p\}$, we have 
\[
\angle (\nabla_x f, \uparrow_x^p ) < \varepsilon 
\text{ and } |\nabla_x f, \uparrow_x^p \!| < \sqrt 2 \varepsilon.
\]
\end{lemma}
\begin{proof}
By Proposition \ref{distance function from boundary}, we have 
\[
d f_x (\uparrow_x^p) > \cos \varepsilon. 
\]
By the definition of the gradient, we obtain 
\[
d f_x (\uparrow_x^p) \le %%%% \langle \nabla_x f , \uparrow_x^p \rangle =
|\nabla_x f| \cos \angle (\nabla_x f , \uparrow_x^p) \le \cos \angle (\nabla f, \uparrow_x^p).
\]
Therefore, we have $\angle (\nabla_x f, \uparrow_x^p) < \varepsilon$.

Since $f$ is $1$-Lipschitz, $|\nabla f| \le 1$. 
And, by the above inequality,  
\[
|\nabla f|_x = \max_{\xi \in \Sigma_x} d f_x (\xi) \ge d f (\uparrow_x^p) > \cos \varepsilon.
\]
Then, we obtain 
\[
|\nabla_x f, \uparrow_x^p|^2 
< |\nabla f|^2 + 1 - 2 |\nabla f| \cos \varepsilon 
\le 2 \sin^2 \varepsilon.
\]
Therefore, $|\nabla f, \uparrow_x^p| < \sqrt 2 \varepsilon$.
\end{proof}

Let us consider the gradient flow $\Phi_t$ of $f = d(S, \cdot)$.
\begin{lemma} \label{upper bound of time}
%%%% For $x \in B(p, \delta R)$, $\Phi_t (x) \in B(p, \delta R)$.
%%%% There exists an explicit constant $C > 0$ such that, 
For every $x \in B(p, \delta R)$, 
\[
|\Phi_t (x), p| \le |x, p| - \cos \varepsilon \cdot t,
\]
whenever this formula has meaning.
In particular, for any $t \ge \delta R / \cos \varepsilon$, we have $\Phi_t(x) = p$.
\end{lemma}
\begin{proof}
Let us set $\gamma(t) = \Phi_t(x)$ the gradient curve for $f$ starting from $\gamma(0) = x$.
If $\gamma(t_0) \neq p$, then
\[
\left. \frac{d}{d t} \right|_{t = t_0+} |\Phi_t (x), p| = - \langle \nabla_{\gamma(t_0)} f, \uparrow_{\gamma(t_0)}^p \rangle < - \cos \varepsilon. %%%% \le -1 + \varepsilon^2.
\]
Integrating this, we have 
\[
|\Phi_{t_0} (x), p| - |x, p| \le -\cos \varepsilon \cdot t_0.
\]
This completes the proof.
\end{proof}

Finally, we estimates the Lipschitz constant of the flow $\Phi$ on $B(p, \delta R)$.
Let us recall that $f$ is $\lambda$-concave on $B(p, \delta R)$ for some $\lambda$.
By Lemma \ref{1-kf concave}, $\lambda$ can be given as follows.
\[
\frac{\cosh (f)}{\sinh (f)} \le \frac{\cosh R}{\sinh(R(1 - \delta))} = \lambda.
\]
By Proposition \ref{gradient flow is Lipschitz}, for any $x, y \in B(p, \delta R)$, 
\[
|\Phi(x,t), \Phi(y,t)| \le e^{\lambda t} |x y|.
\]
Since $f$ is $1$-Lipschitz, for $x \in B(p, \delta R)$ and $t' < t$, we have 
\[
|\Phi(x,t), \Phi(x,t')| \le \int_{t'}^t \left| \frac{d}{d s}^+ \Phi(x, s) \right| d s = \int_{t'}^t |\nabla f|(\Phi(x,s)) d s \le t -t'.
\]
Therefore, we obtain the following. 
\begin{lemma} \label{construction of Lipschitz homotopy}
For any $x, y \in B(p, \delta R)$ and $t \ge s \ge 0$,
\begin{align*}
|\Phi(x,s), \Phi(y,t)| %%%% &\le |\Phi(x,s), \Phi(y,s)| + |\Phi(y,s), \Phi(y,t)| \\
&\le e^{\lambda s} |x,y| + t -s .
\end{align*}
%%%In particular, $\Phi$ is $\sqrt{2} e^{\lambda \ell}$-Lipschitz on the product $B(p, \delta R) \times [0,\ell]$ equipped with the product metric.
\end{lemma}

Note that, by Lemma \ref{upper bound of time}, setting $\ell = \delta_0 R / \cos \varepsilon$, $e^{\lambda \ell}$ can be bounded from above by a constant arbitrary close to $1$ if we choose $\delta_0$ and $R$ so small.

By Lemma \ref{construction of Lipschitz homotopy}, we obtain a Lipschitz homotopy 
\[
\varphi : B(p, \delta_0 R) \times [0, 1] \to B(p, \delta_0 R)
\]
with $\varphi(\cdot,1)= p$
%% and with the Lipschitz constant $\le \sqrt 2 e^{\lambda \ell}$, 
defined by $\varphi(x,t) = \Phi(x, \ell t)$ for $(x,t) \in B(p, \delta_0 R) \times [0,1]$. 

%%%______
\section{Proof of applications} \label{proof of applications}

\subsection{Proof of Corollaries \ref{Lipschitz homotopy class is homotopy class} and \ref{Lipschitz homotopy group is homotopy group}}
Let $V$ be a metric space and $U$ be a subset of $V$ and $p \in V$.
We say that $U$ is {\it Lipschitz contractible to $p$ in $V$} if there exists a Lipschitz map 
\[
h : U \times [0,1] \to V
\]
such that 
\[
h(x,0) = x \text{ and } h(x,1) = p
\]
for any $x \in U$.
We call such an $h$ a {\it Lipschitz contraction} from $U$ to $p$ in $V$.
We say that $U$ is {\it Lipschitz contractible in} $V$ if 
$U$ is Lipschitz contractible to some point in $V$.

%%%% The following lemma is easy.
\begin{lemma} \label{lemma}
Let $U$ be Lipschitz contractible in a metric space $V$.
For any Lipschitz map $\varphi : S^{n-1} \to U$, there exists a Lipschitz map $\tilde \varphi : D^n \to V$ such that $\tilde \varphi |_{ S^{n-1}} = \varphi$. 
%%Here, $ D^n$ is an $n$-dimensional unit closed ball with the Euclidean metric. 
\end{lemma}
\begin{proof}
By the definition, there exist $p \in V$ and a Lipschitz map 
\[
h : U \times [0,1] \to V
\]
such that 
\[
h(x,0) = x \text{ and } h(x,1) = p
\]
for any $x \in U$.
We define a map 
\[
\varphi_1 : S^{n-1} \times [0,1] \to V
\]
by $\varphi_1 = h \circ (\varphi \times \mathrm{id})$.
%%%% \[
%%%% \varphi_1( v, t ) = h(\varphi(v), t)
%%%% \]
%%%% for $v \in S^{n-1}$ and $t \in [0,1]$.
Then, $\varphi_1$ is Lipschitz with Lipschitz constant $\le \Lip\,(h) \cdot \max\{1, \Lip\,(\varphi)\}$.
We define a map 
\[
\varphi_2 : D^n \times \{1\} \to V
\]
by $\varphi_2 (v, 1) = p$ for all $v \in D^n$.
And we consider a space
\[
Y = S^{n-1} \times [0,1] \cup D^n \times \{1\}
\]
equipped with length metric with respect to a gluing $ S^{n-1} \times \{1\} \ni (v,1) \mapsto (v,1) \in \partial D^n \times \{1\}$.
Now we define a map 
\[
\varphi_3 : Y \to V
\]
by 
\[
\varphi_3 = \left\{
\begin{aligned}
&\varphi_1 \text{ on } S^{n-1} \times [0,1] \\
&\varphi_2 \text{ on } D^n \times \{1\}
\end{aligned}
\right.
\]
This is well-defined.
Then, $\varphi_3$ is $\Lip\, (\varphi_1)$-Lipschitz.
Indeed, for $x \in S^{n-1} \times [0,1]$ and $y \in D^n \times \{1\}$, we have
\begin{align*}
|\varphi_3(x), \varphi_3(y)| = |\varphi_3 (x), p|.
\end{align*}
Let $\bar x \in S^{n-1} \times \{1\}$ be the foot of perpendicular segment from $x$ to $ S^{n-1} \times \{1\}$. %%%% such that $|x,\bar x| = |x, S^{n-1} \times \{1\}|$.
We note that $|x, \bar x| \le |x, y|$ and $\varphi_3(\bar x) = p$.
Then, we obtain 
\begin{align*}
& |\varphi_3 (x), p| = |\varphi_3(x), \varphi_3(\bar x)| 
= |\varphi_1(x), \varphi_1(\bar x)| \\
& \hspace{45pt} \le \Lip\, (\varphi_1) |x, \bar x| \le \Lip\,(\varphi_1) |x, y|.
\end{align*}

Obviously, there exists a bi-Lipschitz homeomorphism 
\[
f : D^n \to Y
\]
with $f(0) = (0, 1) \in D^n \times \{1\}$
preserving the boundaries in the sense that it satisfies $f(v) = (v, 0) \in S^{n-1} \times \{0\}$ for any $v \in S^{n-1}$.
Then, we obtain a Lipschitz map $\tilde \varphi := \varphi_3 \circ f$ satisfying the desired condition.
\end{proof}

\begin{definition} \upshape
We say that a metric space $Y$ is a {\it Lipschitz simplicial complex} if there exists a triangulation $T$ of $Y$ satisfying the following. 
For each simplex $S \in T$, there exists a bi-Lipschitz homeomorphism $\varphi_S : \triangle^{\dim S} \to S$. 
Here, the simplex $\triangle^{\dim S}$ is a standard simplex equipped with the Euclidean metric and $S$ has the restricted metric of $Y$.
And, we say that such a triangulation $T$ is a {\it Lipschitz triangulation} of $Y$.
The dimension of $Y$ is given by $\dim Y = \sup_{S \in T} \dim S$.
We only deal with $Y$ of $\dim Y < \infty$.

A Lipschitz simplicial complex $Y$ is called {\it finite} if it has a Lipschitz triangulation consisting of finitely many elements.
%%%% If $A$ is a subcomplex of $Y$ with respect to some Lipschitz triangulation, then $(Y, A)$ is called a {\it Lipschitz simplicial complex pair}.
\end{definition}

%%%% We do not assume that $Y$ is locally finite.
Note that a subdivision, for instance, the barycentric one, of a Lipschitz triangulation is also a Lipschitz triangulation. 

%%%% \newpage
\begin{proposition} \label{HL} %%%% homotopic to Lipschitz map
Let $X$ be an SLLC space, $Y$ be a Lipschitz simplicial complex and $f : Y \to X$ be a continuous map.
%%%% Setting an open subset $U_E$ of $X$ with $f(E) \subset U_E$ for any simplex $E \in T$, 
Then, there exists a homotopy 
\[
h : Y \times [0,1] \to X
\]
from $h_0 = f$ such that $h_1$ is Lipschitz on each simplex of $Y$.

Further, if $f$ is Lipschitz on a subcomplex $A$ of $Y$, then a homotopy $h$ can be chosen so that it is relative to $A$. 
Namely, it satisfies $h(a,t) = a$ for any $a \in A$ and $t \in [0,1]$.
\end{proposition}
\begin{proof}
If $\dim Y = 0$, then we set $h(x, t) = f(x)$ for $x \in Y$ and $t \in [0,1]$.
Then, $h$ is the desired homotopy.

We assume that the assertion holds for $\dim Y \le k-1$.
First, we prove that for any $f : \triangle^k \to X$, there exists a homotopy 
\[
h : \triangle^k \times [0,1] \to X
\]
from $h_0 = f$ to a Lipschitz map $h_1$.
Taking a subdivision if necessary, let us take a finite Lipschitz triangulation $T$ of $\triangle^k$ satisfying the following.
For any $k$-simplex $E \in T$, there exists an open subset $U_E$ of $X$ which is a Lipschitz contractible ball such that $f(E) \subset U_E$.
For any simplex $F \in T$ of $\dim F \le k-1$, we set 
\[
U_F = \bigcap_{F \subset E \in T} U_E.
\]
This is an open subset of $X$. 
Let us denote by $Z$ a $(k-1)$-skeleton of $\triangle^k$ with respect to $T$.
By an inductive assumption, there exists a homotopy 
\[
h : Z \times [0,1] \to X
\]
from $h_0 = f|_Z$ such that for every simplex $F$ of $Z$, 
\begin{itemize}
\item $h_1|_F$ is Lipschitz;
\item $h(F \times [0,1]) \subset U_F$; 
\item if $f|_F$ is Lipschitz, then $h_t |_F = f|_F$ for any $t$.
\end{itemize}
Let $E$ be a $k$-simplex of $\triangle^k$ with respect to $T$.
We denote %%%% by $f^E$ the restriction of $f$ to $E$ and 
by $h^{\partial E}$ the restriction of $h$ to $\partial E \times [0,1]$.
%%%% For every $(k-1)$-simplex $F$ of $\partial E$, we denote by $h^F$ the restriction of $h$ to $F \times [0,1]$.
Then, the image of $h^{\partial E}$ is contained in $\bigcup_{T \ni F \subset \partial E} U_F \subset U_E$.
Since a pair $(E, \partial E)$ has the homotopy extension property, there exists a homotopy 
\[
h^E : E \times [0,1] \to U_E
\]
from $f|_E$ which is an extension of $h^{\partial E}$.
Then, $h^E_1$ is Lipschitz on $\partial E$. 
For another $k$-simplex $E'$ of $\triangle^k$ with common face $E \cap E'$, 
\[
h^E_t = h^{E'}_t
\]
on $E \cap E'$ for all $t$.
Since $U_E$ is Lipschitz contractible ball, by Lemma \ref{lemma}, there is a homotopy 
\[
\bar h^E : E \times [0,1] \to X
\]
relative to $\partial E$ from $\bar h^E_0 = h^E_1$ to a Lipschitz map $\bar h^E_1 : E \to X$.
Let us define a homotopy $\hat h^E : E \to X$ by 
\[
\hat h^E (x,t) = \left\{
\begin{aligned}
h^E(x,t) \text{ if } t \in [0,1/2]; \\
\bar h^E(x,t) \text{ if } t \in [1/2,1].
\end{aligned}
\right.
\]
And we define $\hat h : \triangle^k \times [0,1] \to X$ by 
\[
\hat h (x, t) = \hat h^E (x,t)
\]
for $x \in E \in T$.
Then, $\hat h_0 = f$ and $\hat h_1$ is Lipschitz. 

Next, we consider a continuous map $f : Y \to X$ from a Lipschitz simplicial complex $Y$ of $\dim Y = k$.
%%%% Let us take a Lipschitz triangulation $T$ of $Y$ and an open subset $U_E$ of $X$ with $f(E) \subset U_E$ for each $E \in T$.
Let $Z$ be a $(k-1)$-simplex of $Y$.
By an inductive assumption, there exists a homotopy 
\[
h : Z \times [0,1] \to X
\]
from $h_0 = f|_Z$ and $h_1$ is Lipschitz on every simplex of $Z$.
From now on, let us denote by $E$ a $k$-skeleton of $Y$.
For any $E  \subset Y$, %%%% we denote by $h^{\partial E}$ the restriction of $h$ to $\partial E \times [0,1]$.
by using the homotopy extension property for $(E, \partial E)$ and Lemma \ref{lemma}, we obtain a homotopy 
\[
h^E : E \times [0,1] \to X
\]
which is an extension of $h|_{\partial E \times [0,1]}$ with $h^E_0 = f|_E$. %%%% and $h^E_1$ is Lipschitz.
Since $h^E_1|_{\partial E} = h_1|_{\partial E}$ is Lipschitz, there exists a homotopy 
\[
\bar h^E : E \times [0,1] \to X
\]
relative to $\partial E$ from $\bar h^E_0 = h^E_1$ to a Lipschitz map $\bar h^E_1$.
We set $\bar h(x,t) = h(x,1)$ for $x \in Z$ and $t \in [0,1]$.
And, we define a homotopy $\hat h : Y \times [0,1] \to X$ by 
\[
\hat h(x,t) = \left\{
\begin{aligned}
&h(x, 2 t) &&\text{ if } x \in Z \text{ and } t \in [0,1/2] \\
&\bar h(x, 2 t -1) &&\text{ if } x \in Z \text{ and } t \in [1/2,1] \\
&h^E (x, 2 t) &&\text{ if } x \in E \subset Y \text{ and } t \in [0, 1/2]  \\
&\bar h^E(x, 2 t -1) &&\text{ if } x \in E \subset Y \text{ and } t \in [1/2, 1]
\end{aligned}
\right.
\]
Then, $\hat h_0 = f$ and $\hat h_1$ is Lipschitz on every simplex.
%%%% 2012.0522.1800
\end{proof}

\begin{corollary} \label{HL1}
Let $Y$ be a Lipschitz simplicial complex, $X$ be an SLLC space and $f : Y \to X$ be a continuous map. 
Let $T$ be a Lipschitz triangulation of $Y$ and $\{U_F \,|\, F \in T\}$ be a family of open subsets of $X$ with the following property.
\begin{itemize}
\item $f(F) \subset U_F$ for $F \in T$;
\item $U_F \subset U_E$ for $F$, $E \in T$ with $F \subset E$.
\end{itemize}
Then, there exists a homotopy $h : Y \times [0,1] \to X$ from $h_0 = f$ 
such that for every $F \in T$, 
\begin{itemize}
\item $h_1$ is Lipschitz on $F$; 
\item $h(F \times [0,1]) \subset U_F$; 
\item if $f$ is Lipschitz on $F$, then $h_t = f$ on $F$ for all $t$.
\end{itemize}
\end{corollary}
For instance, fixing $\varepsilon > 0$ and setting $U_F$ an $\varepsilon$-neighborhood of $f(F)$ for every $F \in T$, the family $\{ U_F \,|\, F \in T \}$ satisfies the assumption of Corollary \ref{HL1}.

\begin{proof}[Proof of Corollary \ref{HL1}]
If $\dim Y = 0$, the assertion is trivial.
We assume that Corollary \ref{HL1} holds when $\dim Y \le k-1$ for $k \ge 1$.
Let $Y$ be a Lipschitz simplicial complex with $\dim Y = k$ and $T$ be a Lipschitz triangulation of $Y$.
Let us take a family $\{U_F \,|\, F \in T\}$ of open subsets as the assumption of Corollary \ref{HL1}.
By inductive assumption, there exists a homotopy 
\[
h : Y^{(k-1)} \times [0,1] \to X
\]
from $h_0 = f|_{Y^{(k-1)}}$ and $h_1$ is Lipschitz on each $F \in T$ of $\dim \le k-1$, and $h_t(F) \subset U_F$ for all $t$.
Let us denote by $E$ a $k$-simplex in $T$.
By Proposition \ref{HL}, there exists a homotopy 
\[
h^E : E \times [0,1] \to U_E
\]
from $h^E_0 = f|_E$ to a Lipschitz map $h^E_1$ such that $h^E_t = h_t$ on $\partial E$ for all $t$.  
Then, a concatenation map 
\[
\hat h (x,t) = \left\{
\begin{aligned}
& h(x,t) &&\text{ if } x \in Y^{(k-1)};\\
& h^E(x,t) &&\text{ if } x \in E.
\end{aligned}
\right.
\]
is a desired homotopy.
\end{proof}

\begin{remark} \upshape \label{rem:LLC}
We note that Proposition \ref{HL} and Corollary \ref{HL1} above can be also proved assuming $X$ is just LLC instead of SLLC.
Here, we say that a metric space $X$ is locally Lipschitz contractible, for short LLC, if for any $p \in X$ and $\varepsilon > 0$, there exist $r \in (0,\varepsilon]$ and a Lipschitz contraction $\varphi$ from $U(p,r)$ to $p$ in $U(p,\epsilon)$. 
We also remark that Corollaries \ref{Lipschitz homotopy class is homotopy class} and \ref{Lipschitz homotopy group is homotopy group} are true if $X$ is just LLC.
\end{remark}

%%%
Let us start to prove Corollaries \ref{Lipschitz homotopy class is homotopy class} and \ref{Lipschitz homotopy group is homotopy group}. 
\begin{proof}[Proof of Corollaries \ref{Lipschitz homotopy class is homotopy class} and \ref{Lipschitz homotopy group is homotopy group}]
Let us take a finite Lipschitz simplicial complex pair $(P, Q)$, possibly $Q$ is empty.
%%And, let us take an SLLC space $X$
We prove Corollaries \ref{Lipschitz homotopy class is homotopy class} and \ref{Lipschitz homotopy group is homotopy group} assuming $X$ to be SLLC.
%%%%an Alexandrov space $X$ 
Let $A$ be an open subset in $X$. 
Let us consider a continuous map $f : (P,Q) \to (X,A)$.
By Corollary \ref{HL1} and Theorem \ref{Alex sp is SLLC}, we obtain a homotopy 
\[
\varphi : (P,Q) \times [0,1] \to (X,A) 
\]
from $\varphi_0 = f$ to a Lipschitz map $\varphi_1 : (P,Q) \to (X,A)$.
Here, we note that since $A$ is open in $X$, the homotopy $\varphi_t$ can be chosen so that $\varphi_t(Q) \subset A$.
Then, we obtain a corresponding 
\begin{equation} \label{corresponding}
C((P,Q), (X,A)) \ni f \mapsto \varphi_1 \in \mathrm{Lip}((P,Q), (X,A)),
\end{equation}
where $C(\ast, \ast\ast)$ (resp. $\mathrm{Lip}(\ast, \ast\ast)$) denotes the set of all continuous (resp. Lipschitz) maps from $\ast$ to $\ast\ast$.

Let us consider two continuous maps $f$ and $g$ from $(P,Q)$ to $(X,A)$ such that they are homotopic to each other. 
From the correspondence \eqref{corresponding}, we obtain Lipschitz maps $f'$ and $g'$ from $(P,Q)$ to $(X,A)$ which are homotopic to $f$ and $g$, respectively.
Connecting these homotopies, we obtain a homotopy 
\[
H : (P,Q) \times [0,1] \to (X,A)
\]
between $H(\cdot, 0) = f'$ and $H(\cdot,1) = g'$.
Now, we consider a Lipschitz simplicial complex $\tilde P = P \times [0,1]$ and a subcomplex $\tilde R = P \times \{0,1\}$. 
Then, the map $H$ is Lipschitz on $\tilde R$.
Hence, by Proposition \ref{HL}, we obtain a homotopy 
\[
\tilde H : \tilde P \times [0,1] \to X 
\]
relative to $\tilde R$ from $\tilde H(\cdot, 0) = H$ to a Lipschitz map $\tilde H(\cdot, 1)$.
Then, $\tilde H(\cdot,1)$ is a Lipschitz homotopy between $f'$ and $g'$.
Therefore, we conclude that the corresponding \eqref{corresponding} sends a homotopy to a Lipschitz homotopy.
It completes the proof of Corollary \ref{Lipschitz homotopy class is homotopy class}.

Let us consider a pointed $n$-sphere $(S^n, p_0)$ and an Alexandrov space $X$ with point $x_0 \in X$.
Then, for any map $f : (S^n, p_0) \to (X, x_0)$, the restriction $f|_{\{p_0\}}$ is always Lipschitz.
Hence, by an argument as above and Proposition \ref{HL}, we obtain the conclusion of Corollary \ref{Lipschitz homotopy group is homotopy group}.
\end{proof}

\subsection{Plateau problem}

We first recall the definition of the Sobolev space of metric space target, to state the setting of Plateau problem in an Alexandrov space as in the introduction, referring \cite{KS} and \cite{MZ}.
For a complete metric space $X$ and a domain $\Omega$ in a Riemannian manifold having compact closure, a function $u : \Omega \to X$ is said to be {\it $L^2$-map} if $u$ is Borel measurable and for some (and any) point $p_0 \in X$, the integral 
\[
\int_\Omega |u(x), p_0|^2 d \mu 
\]
is finite, where $\mu$ is the Riemannian volume measure.
The set of all $L^2$-maps from $\Omega$ to $X$ denotes $L^2(\Omega,X)$.
We recall the definition of the energy of $u \in L^2(\Omega,X)$.
For any $\varepsilon > 0$, we set $\Omega_\varepsilon = \{x \in \Omega \,|\, d(\partial \Omega, x) > \varepsilon \}$ and define an approximate energy density $e_\varepsilon^u : \Omega_\varepsilon \to \mathbb R$ by 
\[
e_\varepsilon^u (x) = \frac{1}{\omega_n} \int_{S(x,\varepsilon)} \frac{d(u(x),u(y))^2}{\varepsilon^2} \frac{d \sigma}{\varepsilon^{n-1}}.
\]
Here, $n = \dim \Omega$, $S(x,\varepsilon)$ is the metric sphere around $x$ with radius $\varepsilon$ and $\sigma$ is the surface measure on it.
By \cite[1.2iii]{KS}, we obtain 
\[
\int_{\Omega_\varepsilon} e_\varepsilon^u (x)\,\! d \mu \le C \varepsilon^{-2}.
\]
Let us take a Borel measure $\nu$ on the interval $(0,2)$ satisfying 
\[
\nu \ge 0, \nu ((0,2)) = 1, \int_0^2 \lambda^{-2} d \nu (\lambda) < \infty.
\]
An averaged approximate energy density ${}_\nu e_\varepsilon^u(x)$ is defined by 
\[
{}_\nu e_\varepsilon^u(x) = \left\{
\begin{aligned}
&\int_0^2 e^u_{\lambda \varepsilon} (x)\,\! d \nu(\lambda) &&\text{ if } x \in \Omega_{2 \varepsilon}\\
&0 &&\text{ otherwise }
\end{aligned}
\right.
\]
Let $C_c(\Omega)$ be the set of all continuous function on $\Omega$ with compact support.
We define a functional $E_\varepsilon^u : C_c(\Omega) \to \mathbb R$ by 
\[
E_\varepsilon^u (f) := \int_\Omega f(x) {}_\nu e_\varepsilon^u d \mu(x).
\]
Then, the {\it energy} of $u$ is defined by 
\[
E^u = \sup_{f \in C_c(\Omega), 0 \le f \le 1} \limsup_{\varepsilon \to 0} E_\varepsilon^u(f).
\]
The $(1,2)$-Sobolev space is defined as 
\[
W^{1,2}(\Omega, X) = \{u \in L^2(\Omega, X) \,|\, E^u < \infty \}.
\]

We start to prove Corollary \ref{corollary: Plateau problem}.
\begin{proof}[Proof of Corollary \ref{corollary: Plateau problem}]
Let $\Gamma$ be a rectifiable closed Jordan curve in an Alexandrov space $X$, which is  toplogically contractible.
Since the rectifiability of $\Gamma$, we can take a Lipschitz monotonic parametrization 
\[
\gamma : S^1 \to \Gamma.
\]
By the contractibility of $\Gamma$, there exists a continuous map 
\[
h : \Gamma \times [0,1] \to X
\]
such that $h(\cdot, 0) = id_\Gamma$ and $h(\cdot, 1) = p$ for some $p \in X$.
We define a map $f : S^1 \times [0,1] \to X$ by $f(x,t) = h(\gamma(x),t)$.
Further, we set $f(y,1) = p$ for $y \in D^2$.
By taking reparametrization of $f : S^1 \times [0,1] \cup D^2 \times \{1\} \to X$, we obtain a continuous map 
\[
g : D^2 \to X
\]
such that $g|_{\partial D^2} = \gamma$.

By Proposition \ref{HL}, there exists a homotopy 
\[
\tilde h : D^2 \times [0,1] \to X
\]
relative to $\partial D^2$ such that $\tilde h(\cdot,0) = g$ and $\tilde h(\cdot,1)$ is Lipschitz.
Thus, we obtain the Lipschitz map $\tilde g = \tilde h(\cdot,1)$ such that $\tilde g|_{\partial D^2} = \gamma$.
By the definition of the energy, we obtain 
\[
E(\tilde g) \le \mathrm{Lip}(\tilde g)^2 < \infty.
\]
Here, $\mathrm{Lip}(\tilde g)$ is the Lipschitz constant of $\tilde g$.
Therefore, we conclude $\tilde g \in \mathcal F_\Gamma$.
\end{proof}

\section{A note on the infinite dimensional case} \label{sec:infinite dim}
It is known that the (Hausdorff) dimension of an Alexandrov space is nonnegative integer or infinite. 
%%%Studies of {\it infinite dimensional} Alexandrov spaces are few. 
There are only few works of infinite dimensional Alexandrov spaces.
It is not known whether an infinite dimensional Alexandrov space is locally contractible. 

When we consider an Alexandrov space of {\it possibly infinite dimension}, we somewhat generalize Definition \ref{def:Alexandrov space} as follows.
A complete metric space $X$ is called an {\it Alexandrov space} %%of curvature} $\ge -1$ 
if it is a length metric space and satisfies the quadruple condition locally. %%modeled on the hyperbolic plane.
Here, a complete metric space $X$ is {\it length} if every two points $p, q \in X$ and any $\varepsilon > 0$, there exists a point $r \in X$ satisfying $\max \{|pr|,  |rq| \} < |pq| /2 + \varepsilon$. 
Since a length metric space has no geodesic in general, to define a notion of a lower curvature bound, we change the triangle comparison condition by the quadruple condition. 
Here, an open subset $U$ of a length space $X$ satisfies the {\it quadruple condition} modeled on the $\kappa$-plane $\mathbb M_\kappa$ if for every distinct four points $p_0, p_1, p_2$ and $p_3$ in $U$, we obtain 
\[
\tilde \angle p_1 p_0 p_2 + \tilde \angle p_2 p_0 p_3 + \tilde \angle p_3 p_0 p_1 \le 2 \pi,
\]
where $\tilde \angle = \tilde \angle_\kappa$ denotes the comparison angle modeled on $\mathbb M_\kappa$.

By the standard argument, any geodesic triangle (if it exists) in an Alexandrov space of possibly infinite dimension satisfies the triangle comparison condition.
It is known that finite dimensional Alexandrov spaces are proper metric space, in particular, by Hopf-Rinow theorem, they are geodesic spaces.
%%%Thus, dealing with only finite dimensional Alexandrov spaces, a generalized definition as above is not need.

Plaut \cite{Plaut} proved that an Alexandrov space of infinite dimension is an ``almost'' geodesic space. Presicely, 
\begin{theorem}[\cite{Plaut}] \label{theorem: Plaut}
Let $X$ be an Alexandrov space of infinite dimension. 
For any $p \in X$, a subset $J_p \subset X$ defined by 
\[
J_p = \bigcap_{\delta > 0} \{q \in X - \{p\} \,|\, \text{there exists } x \in X - \{p, q\} \text{ with } \tilde \angle p q x > \pi - \delta\}
\]
is dense $G_\delta$ subset in $X$, and for every $q \in J_p$, there exists a unique geodesic connecting $p$ and $q$.
\end{theorem}

We now show that the compactness of the space of directions at some point implies the Lipschitz contractibility around the point. 
%%%% The first author was telled by Takumi Yokota at the conference ``The fourth geometry meeting'' in St Petursberg that there exists an example of infinite dimensional compact Alexandrov space of curvature $\ge 1$, which was constructed by Alexander Lythack. 

\begin{proposition}
Let $X$ be an Alexandrov space of infinite dimension. 
Suppose that there exists a point $p \in X$ such that the space of directions $\Sigma_p$ at $p$ is compact. 
Then, the following are true. 
\begin{itemize}
\item[(i)] The pointed Gromov-Hausdorff limit of scaling space $(r X, p)$ as $r \to \infty$ exists and it is isometric to the cone over $\Sigma_p$.
\item[(ii)] $\Sigma_p$ is a geodesic space.
\item[(iii)] $X$ is proper.
\item[(iv)] There exists $R_0 > 0$ depending on $p$ such that for every $R \le R_0$, $U(p, R)$ is Lipschitz contractible to $p$ in itself.
\end{itemize}
\end{proposition}
\begin{proof}
%%%% Since $\Sigma_p$ is compact, for any $\varepsilon > 0$, there exists a maximal $\varepsilon$-discrete subset  $\{v_\alpha\}_\alpha \subset \Sigma_p$ consisting of finitely many vectors. %%%% such that the index set $\{\alpha\}$ is finite.
(i). Let $K = K(\Sigma_p)$ be the Euclidean cone over $\Sigma_p$ and $B$ be the unit ball around the origin $o$.
%%%Let $J_p$ be the set of all points which can connect to $p$ by a unique geodesic. 
Let $J_p$ be the set defined in Theorem \ref{theorem: Plaut}.
%%%%Then, by Theorem \ref{theorem: Plaut}, $J_p$ contains a dense $G_\delta$ subset. 
For any $\varepsilon > 0$, we take a finite $\varepsilon$-net $\{v_\alpha\}_\alpha \subset B$. 
We may assume that every $v_\alpha$ is contained in $K(\Sigma_p') - \{o\}$.
Namely, there exists $r > 0$ such that for every $\alpha$, there is a geodesic $\gamma_\alpha$ staring from $p$ having the direction $\frac{v_\alpha}{|v_\alpha|}$ with length at least $r$.
Let $x_\alpha \in B(p, r)$ be taken as $x_\alpha = \gamma_\alpha (r |v_\alpha|)$.
Then, $\{x_\alpha\}_\alpha$ is an $\varepsilon$-net in $\frac{1}{r} B(p, r)$.
Indeed, for any $x \in B(p, r) \cap J_p$, setting $v = \log_p (x) \in K(\Sigma_p)$, and then $\frac{1}{r} v \in B$.
Then, there exists $\alpha$ such that $|v_\alpha, \frac{1}{r}v| \le \varepsilon$.
Therefore, $|r v_\alpha, v| \le r \varepsilon$.
%%%% If a lower curvature bound $K_X$ of $X$ is not less than$0$, then $\exp_p$ is $1$-Lipschitz. 
%%%% If $K_X < 0$, then 
We may assume that a lower curvature bound of $X$ is less than or equals to $0$. 
Then 
\[
\exp_p : B(o, r) \cap \mathrm{dom} (\exp_p) \to B(p, r)
\]
is $1$-Lipschitz, where $\mathrm{dom} (\exp_p)$ is the domain of $\exp_p$.
Therefore, $|x_\alpha, x|_X \le r \varepsilon$.

Let us retake $r$ to be small so that 
\[
\left| \frac{|x_\alpha, x_\beta|}{r} - |v_\alpha, v_\beta| \right| \le \varepsilon.
\]
Then, the map $v_\alpha \mapsto x_\alpha$ implies a $C\varepsilon$-approximation between $B$ and $\frac{1}{r} B(p, r)$ for any small $r$. 
Here, $C$ is a constant not depending on any other term.
Therefore, the pointed space $(\frac{1}{r} X, p)$ is Gromov-Hausdorff converging to $(K(\Sigma_p), o)$ as $r \to 0$.

(ii) obviously holds by (i) and (iii).
We prove (iii). 
Let us consider any closed ball $B(p, r)$ centered at $p$. 
Let us  take any sequence $\{x_i\} \subset B(p, r)$.
We take $y_i \in B(p, r) \cap J_p$ such that $|x_i, y_i| \le 1 / i$.
Then, $v_i = \log_p (y_i) \in B(o, r) \subset T_p X$ is well-defined.
By (i), $T_p X$ is proper.
Hence, there exists a converging subsequence $\{v_{n(i)}\}_i$ of $\{v_i\}_i$.
Since $\exp_p$ is Lipschitz, $\{x_{n(i)}\}$ is converging.

We recall that the proof of Theorem \ref{Alex sp is SLLC} started from the assertion \eqref{eq 0} in Proposition \ref{distance function from boundary}. %%%% !!!
The assertion (i) guarantees \eqref{eq 0}.
Therefore, one can prove (iv) in the same way as the proof of Theorem \ref{Alex sp is SLLC}.
\end{proof}

\section{An estimation of simplicial volume of Alexandrov spaces} \label{sec:simp vol}
In this section, we consider an Alexandrov space having a lower {\it Ricci} curvature bounds, 
and we prove an estimation of the simplicial volume of such a space as stated in Theorem \ref{thm:simp vol}. 
The original form of Theorem \ref{thm:simp vol} was proved by Gromov \cite{G} when $X$ is a Riemannian manifold with a lower Ricci curvature bound. %%% by a negative constant.

The original Gromov's proof was depending on the well-known Bishop-Gromov volume inequality.
For an Alexandrov space of curvature $\ge \kappa$ by some $\kappa \in \mathbb R$, its Hausdorff measure is known to satisfy the Bishop-Gromov type volume growth estimate.
The second author's proof of Corollary \ref{cor:simp vol} was depending on this volume growth estimate (\cite{Y simp}).
% Therefore, if some assumption implies such a volume growth estimate (with respect to some reference measure) on an Alexandrov space, then the proof given in \cite{G} and \cite{Y simp}  also works.
It is known that all several natural generalized notions of a lower Ricci curvature bound induce a volume growth estimate. 
Among them, the condition named local reduced curvature-dimension condition introduced by Bacher and Sturm (\cite{BS}) %%%is sufficient and useful 
can be used %%%to induce a volume growth estimate in our situation and 
to prove Theorem \ref{thm:simp vol}.
For completeness, we explain it as follows.

%%%We prove an inequality between the simplicial volume and the volume of such spaces.

% In this section, we just say that a complete metric space $X$ is an Alexandrov space if $X$ is a geodesic space and any point in $X$ has a neighborhood which has curvature $\ge \kappa$ for some $\kappa \in \mathbb R$ in the sense of $(2)$ of Definition \ref{def:Alexandrov space}.
% When $X$ is compact, there exists $\kappa_0 \in \mathbb R$ such that $X$ has curvature $\ge \kappa_0$.

\subsection{Several conditions of lower Ricci curvature bound}
We recall several generalized notions of a lower bound of Ricci curvature defined on a pair of a metric space and a Borel measure on it. 
For their theory, history and undefined terms appearing in the following, we refer \cite{S}, \cite{S2}, \cite{BS}, \cite{CS}, \cite{O} and their references.

In this section, we denote by $M$ a complete separable metric space.
By $\mathcal P_2(M)$ we denote the set of all Borel probability measures $\mu$ on $M$ with finite second moment. %%, which means that $\int_M d(x,x_0)^2 \, d \mu < +\infty$ for some (hence all) $x_0 \in M$.
A metric called the $L_2$-Wasserstein distance $W_2$ is defined on $\mathcal P_2(M)$. 
%%%It is known that $\mathcal P_2(M)$ equipped with the $L_2$-Wasserstein metric is a complete separable metric space and that $M$ is geodesic if and only if so is $\mathcal P_2(M)$.
Let us fix a locally finite Borel measure $m$ on $M$. 
Such a pair $(M,m)$ is called a metric measure space. 
%%%Let us denote by $\mathcal P_2(M,m)$ the subset of $\mathcal P_2(M)$ consisting of all measures which are absolutely continuous with respect to $m$ 
%%%and by $\mathcal P_\infty(M,m)$ the subset of $\mathcal P_2(M,m)$ consisting of all measures having bounded support.
Let us denote by $\mathcal P_\infty(M,m)$ the subset of $\mathcal P_2(M)$ consisting of all measures which are absolutely continuous in $m$ and have bounded support.

From now on, $K$ and $N$ denote real numbers with $N \ge 1$. 
For $\nu \in \mathcal P_\infty(M,m)$ with density $\rho = d \nu / d m$, its R\'enyi entropy with respect to $m$ is given by 
\[
S_N(\nu | m) := - \int_M \rho^{1 - 1/N} \, d m = - \int_M \rho^{-1/N} \, d \nu.
\]
For $t \in [0,1]$, a function $\sigma_{K,N}^{(t)} : (0,\infty) \to [0,\infty)$ is defined as 
\[
\sigma_{K,N}^{(t)}(\theta) = \left\{
\begin{array}{ll}
+ \infty &\text{if } K \theta^2 \ge N \pi^2 \\ 
\frac{\sn_{K/N}(t \theta)}{\sn_{K/N}(\theta)} &\text{if else}. 
\end{array}
\right.
\]
And, we set $\tau_{K,N}^{(t)}(\theta) = t^{1/N} \sigma_{K,N-1}^{(t)}(\theta)^{(N-1)/N}$. %%%%for $N >1$.

\begin{definition}[\cite{BS},\cite{CS},\cite{S2}]%%[{\cite[Definition 2.3]{BS}}, {\cite[Definition 2.5]{CS}}] 
\upshape \label{def:CD}
Let $K$ and $N$ be real numbers with $N \ge 1$.
Let $(M,m)$ be a metric measure space.

We say that $(M,m)$ satisfies the {\it reduced curvature-dimension condition} $\cd^\ast(K,N)$ {\it locally} -- denoted by $\cd_\loc^\ast(K,N)$ -- if 
for any $p \in M$ there exists a neighborhood $M(p)$ such that
for all $\nu_0$, $\nu_1 \in \mathcal P_\infty(M,m)$ supported $M(p)$, denoting those densities by $\rho_0$, $\rho_1$ with respect to $m$, there exist an optimal coupling $q$ of $\nu_0$ and $\nu_1$ and a geodesic $\Gamma : [0,1] \to \mathcal P_\infty(M,m)$, parametrized proportionally to arclength, 
connecting $\nu_0 = \Gamma(0)$ and $\nu_1 = \Gamma(1)$ such that 
\begin{align*} \label{eq:CD}
S_{N'} (\Gamma(t) | m) \le 
- \int_{M \times M} 
&\left[
\sigma_{K, N'}^{(1-t)}(d(x_0,x_1)) \rho_0^{-1/N'}(x_0) \right.\\
&\left.
+\, \sigma_{K, N'}^{(t)}(d(x_0,x_1)) \rho_1^{-1/N'}(x_1) \right] \!
d q(x_0,x_1) \nonumber
\end{align*}
holds for all $t \in [0,1]$ and all $N' \ge N$.

%%%When $M(p)$ can be taken to be $M$ for some $p \in M$, we simply say that $(M,m)$ satisfies $\cd^\ast(K,N)$.

We say that $(M,m)$ satisfies the {\it curvature-dimension condition} $\cd(K,N)$ {\it locally} -- denoted by $\cd_\loc(K,N)$ -- if 
it satisfies $\cd_\loc^\ast(K,N)$ with %%$\mathcal P_\infty(M,m)$ and 
$\sigma_{K,N'}^{(s)}$ replaced by %%$\mathcal P_2(M,m)$ and 
$\tau_{K,N'}^{(s)}$ for each $s \in [0,1]$ and $N' \ge N$.

%%%When $M(p)$ can taken to be $M$ for some $p$, we simply say that $(M,m)$ satisfies $\cd(K,N)$.
\end{definition}
The (global) conditions $\cd^\ast(K,N)$ and $\cd(K,N)$ are also defined as similar to and imply corresponding local ones.
%%%Global conditions imply corresponding local conditions.

From the identical inequality $\tau_{K,N}^{(t)}(\theta) \ge \sigma_{K,N}^{(t)}(\theta)$, $\cd(K,N)$ (resp. $\cd_\loc(K,N)$) induces $\cd^\ast(K,N)$ (resp. $\cd_\loc^\ast(K,N)$).
Further, it is known that the local CD-conditions are equivalent in the following sense:   
%%\cite[Proposition 5.5]{BS}. 
When a mathematical condition $\varphi(K)$ is given for each $K \in \mathbb R$, we say that an mathematical object $P$ satisfies $\varphi(K-)$ if $P$ satisfies $\varphi(K')$ for all $K' < K$.

\begin{theorem}[{\cite[Proposition 5.5]{BS}}]
Let $K, N \in \mathbb R$ with $N \ge 1$ and let $(M,m)$ be a metric measure space. 
Then, $(M,m)$ satisfies $\cd_\loc^\ast(K-,N)$ %%%for any $K' < K$ 
if and only if 
it satisfies $\cd_\loc(K-,N)$. %%for any $K' < K$.
\end{theorem}
%%%It follows from the fact that for any $K' < K$, there exists $\theta^\ast > 0$ such that $\tau_{K',N}^{(t)}(\theta) \le \sigma_{K,N}^{(t)}(\theta)$ for all $0 \le \theta \le \theta^\ast$ and all $t \in [0,1]$.

There is another notion of a lower Ricci curvature bound in metric measure spaces which is called the {\it measure contraction property}, denoted by $\mcp(K,N)$.
Since we does not use its theory to prove Theorem \ref{thm:simp vol} in this paper, omit its definition.
For the definition and theory, we refer \cite{O} and \cite{S2}.

% Let $A = (A,\mathcal F_A)$ and $B = (B, \mathcal F_B)$ be measurable spaces. 
% A Markov kernel $Q$ from $A$ to $B$ is a map $Q : A \times \mathcal F_B \to [0,1]$ such that for any $x \in A$, $Q(x,\cdot)$ is a probability measure on $\mathcal F_B$ and for any $B' \in \mathcal F_B$, the map $x \mapsto Q(x,B')$ is measurable.

% For $K, N \in \mathbb R$ with $N \ge 1$ and $t \in (0,1)$, we set 
% \[
% \varsigma_{K,N}^{(t)}(\theta) = t \sigma_{K, N-1}^{(t)}(\theta){}^{N-1} = \tau_{K,N}^{(t)}(\theta){}^N.
% \]
% For $x,y,z$ in a metric space and $t \in [0,1]$, $z$ is called a $t$-intermediate point of $x$ and $y$ if $d(x,z) = t d(x,y)$ and $d(z,y) = (1-t) d(x,y)$ holds.
% \begin{definition}[{\cite[Definition 2.1]{O}}, {\cite[Definition 5.1]{S2}}] \upshape \label{def:MCP}
% Let $K, N \in \mathbb R$ with $N \ge 1$.
% A metric measure space $(M,m)$ satisfies the {\it measure contraction property} $\mcp(K,N)$ if for any $t \in (0,1)$, there exists a Markov kernel $Q_t$ from $M \times M$ to $M$ such that for $m \otimes m$-a.e.\! $(x,y) \in M \times M$ and $Q_t(x,y,\cdot)$-a.e.\! $z \in M$, the point $z$ is a $t$-intermediate point of $x$ and $y$, and such that for $m$-a.e.\! $x \in M$ and for every measurable $B \subset M$, 
% \begin{align*}
% \int_M \varsigma_{K,N}^{(t)} (d(x,y)) Q_t(x,y,B) \, d m(y) &\le m(B), \\
% \int_M \varsigma_{K,N}^{(1-t)} (d(x,y)) Q_t(y,x,B) \, d m(y) &\le m(B).
% \end{align*}
% \end{definition}

A metric measure space $(M,m)$ is called non-branching if $M$ is a geodesic space and is non-branching in the sense that for any four points $x$, $y$, $z_1$, $z_2$ in $M$, if $y$ is a common midpoint of $x$ and $z_1$ and of $x$ and $z_2$, then $z_1 = z_2$.
It is known that a non-branching metric measure space satisfying $\cd(K,N)$ satisfies $\mcp(K,N)$. 
Recently, Cavalleti and  Sturm proved %%that the local curvature-dimension condition implies the measure contraction property for non-branching metric measure spaces.

\begin{theorem}[{\cite[Theorem 1.1]{CS}}] \label{CD_loc to MCP}
Let $(M,m)$ be a non-branching metric measure space. 
Let $K,N \in \mathbb R$ with $N \ge 1$.
If $(M,m)$ satisfies $\cd_\loc(K,N)$, then it satisfies $\mcp(K,N)$.
\end{theorem}

\subsection{Bishop-Gromov volume growth estimate}
Let $(M,m)$ be a metric measure space and $x \in \supp (m)$. 
We set 
\begin{align*}
v_x(r) &:= m(B(x,r)). %%\text{ and } \\
%%%s_x(r) &:= \limsup_{\delta \to 0} \frac{1}{\delta} m(B(x,r+\delta) - U(x, r)).
\end{align*}
For $K, N \in \mathbb R$ with $N > 1$, we define 
\begin{align*}
\bar v_{K,N}(r) &= \int_0^r \sn_{K/(N-1)}^{N-1}(t) \,d t. %%%\text{ and } \\
%%%\bar s_{K,N}(r) &= \frac{d}{d r} \bar v_{K, N}(r) = \sn_{K/(N-1)}^{N-1}(r).
\end{align*}

A metric measure space $(M,m)$ satisfies the {\it Bishop-Gromov volume growth estimate} $\bg(K,N)$ if for any $x \in \supp (m)$, the function
\[
{v_x(r)} /{\bar v_{K,N}(r)} 
%%\frac{v_x(r)}{\bar v_{K,N}(r)} 
%%%%\hspace{1em}\text{and}\hspace{1em} \frac{s_x(r)}{\bar s_{K,N}(r)}
\]
is nonincreasing in $r \in (0, \infty)$, (with $r \le \pi \sqrt{(N-1)/K}$ if $K > 0$).

%%%Note that if $(M,m)$ satisfies $\bg(K,N)$, then $r \mapsto v_x(r)$ is locally Lipschitz, and hence $s_x(r)$ is actually the derivative of $v_x(r)$ for almost every $r$.
Since $\bar v_{K,N}(r)$ is continuous in $K$, $\bg(K-,N)$ implies $\bg(K,N)$.
The Bishop-Gromov volume growth estimate is implied by several lower Ricci curvature bounds, for instance the measure contraction property.
\begin{theorem}[{\cite[Theorem 5.1]{O}}, {\cite[Remark 5.3]{S2}}] \label{BG CD}
If $(M,m)$ satisfies $\mcp(K,N)$, then it satisfies $\bg(K,N)$.
\end{theorem}

Summarizing above facts, we can use the following implication:
%%%For each $K \in \mathbb R$, given a mathematical sentence $\varphi(K)$ about metric measure spaces, we say that $(M,m)$ satisfies $\varphi(K-)$ if it satisfies $\varphi(K')$ for every $K' < K$.
Let $K, N \in \mathbb R$ with $N \ge 1$. %%%and $K' < K$. 
For a non-branching metric measure space $(M,m)$, %% such that $\supp(m)$ is a geodesic space, implications
\begin{equation} \label{implication}
\left\{
\begin{array}{l}
%%\cd_\loc(K,N) \implies 
\cd_\loc^\ast(K,N) \implies \cd_\loc^\ast(K-,N) \iff \cd_\loc(K-,N) \\
\implies \mcp(K-,N) \implies \bg(K-,N) \implies \bg(K,N)
\end{array}
\right.
\end{equation}
holds.

\subsection{Universal covering space with lifted measure}
Let $X$ be a semi-locally simply connected space.
Then, there is a universal covering $\pi : Y \to X$. 
In addition, if $X$ is a length space, then $Y$ is also considered as a length space. %%%together with a length structure $L_Y$ given by $L_Y(\gamma) = L_X (\pi \circ \gamma)$ for any curve $\gamma$ in $Y$, where $L_X$ is the length structure of $X$.
The map $\pi$ becomes a local isometry.

In addition, we assume that $(X,m)$ is a proper metric measure space. %%and $X$ is a geodesic space and is semi-locally simply connected.
Let $\mathcal V$ be the family of all open sheets of the universal covering $\pi : Y \to X$.
We define a set function $m_Y : \mathcal V \to [0, \infty]$ by
\[
m_Y (V) = m (\pi (V)).
\]
One can naturally extend $m_Y$ to a Borel measure on $Y$.
We also write its measure as $m_Y$, and call it the {\it lift} of $m$.
% an outer measure $\bar m_Y : 2^Y \to [0,\infty]$ defined as follows.
% \begin{align*}
% \bar m_Y^\delta(S) &:= \inf \left\{\sum_{j=1}^\infty \mu_Y(V_j) : \{V_j\} \subset \mathcal V, S \subset \bigcup_j V_j \text{ and } \diam V_j \le \delta \right\} \\
%  \bar m_Y (S) &:= \lim_{\delta \to 0} \bar m_Y^\delta(S)
% \end{align*}
% for every $S \subset Y$.
% By the Caratheodory's criterion, $\bar m_Y$ is a Borel measure on $Y$ which is an extension of $m_Y$.
% We will write $m_Y$ instead of $\bar m_Y$ and call it the {\it lift} of $m$. 
Since $m$ is locally finite, so is $m_Y$.

%%For any geodesic $\Gamma : [0,1] \to \mathcal P_2(M)$ between $\Gamma(0) = \nu_0$ and $\Gamma(1) = \nu_1$ and any $0 < t < 1$, setting $A_0 = \supp(\nu_0)$ and $A_1 = \supp(\nu_1)$, $\supp(\Gamma(t))$ is contained in $A_t$ (see \cite[Lemma 2.11]{S}), where %%%$A_t$ is the $t$-intermediate set between $A_0$ and $A_1$:
% \[
% A_t = 
% \left\{
% y \in M \mid 
% \begin{array}{l}
% \exists\, \gamma : [0,1] \to M \text{ a constant speed geodesic } \\
% \text{such that } \gamma(0) \in A_0, \gamma(1) \in A_1 \text{ and } \gamma(t) = y
% \end{array}
% \right\}.
% \]
%%In particular, 
In general, for a geodesic $\Gamma : [0,1] \to \mathcal P_2(M)$, if $\Gamma(0)$ and $\Gamma(1)$ 
are supported on $U(x,r)$ for some $x \in X$ and $r > 0$, then $\Gamma(t)$ is supported on $U(x,2 r)$ for every $t \in (0,1)$ (\cite[Lemma 2.11]{S}).
Therefore, we obtain 
\begin{proposition}[{cf.\! \cite[Theorem 7.10]{BS}}] 
\label{lift} 
The local (reduced) curvature-dimension condition is inherited to the lift.
Namely, let $K, N \in \mathbb R$ with $N \ge 1$ and let $(X,m)$ and $(Y,m_Y)$ be as above.
%%%Suppose that $m$ has full support.
If $(X,m)$ satisfies $\cd_\loc(K,N)$ (resp. $\cd_\loc^\ast(K,N)$), then $(Y,m_Y)$ also satisfies $\cd_\loc(K,N)$ (resp. $\cd_\loc^\ast(K,N)$).
\end{proposition}

\subsection{Proof of Theorem \ref{thm:simp vol}}
\begin{proof}[Proof of Theorem \ref{thm:simp vol}]
Let $X$ be an $n$-dimensional compact orientable Alexandrov space without boundary.
Let $m$ be a locally finite Borel measure on $X$ with full support.
We assume that $(X,m)$ satisfies $\cd_\loc^\ast(K,N)$ for $K < 0$ and $N \ge 1$. 
By Proposition \ref{lift}, the universal covering $Y$ of $X$ with lift $m_Y$ of $m$ also satisfies $\cd_\loc^\ast(K,N)$. 
And, $Y$ is an $n$-dimensional Alexandrov space. 
Since $m$ has full support, so is $m_Y$.
By the implication \eqref{implication}, $(Y,m_Y)$ satisfies $\bg(K,N)$. 
Therefore, as mentioned in the preface of this section, the proof of original Gromov's theorem relying on the Bishop-Gromov volume comparison works in our setting (cf. \cite[\S 2]{G} \cite[Appendix]{Y simp}).
Hence, we can prove Theorem \ref{thm:simp vol}.
We recall such an argument.
For undefined terms appearing and for facts used in the following argument, we refer \cite{G} and \cite{Y simp}.

Let $\mathcal M$ (resp. $\mathcal M_+$) be the Banach space (resp. the set) %%of all finite signed Borel measures and 
of all finite singed (resp. positive) Borel measure on $Y$, equipped with the norm $\| \mu \| = \int_Y \,d |\mu| \in [0,\infty)$ for $\mu \in \mathcal M$.
Due to the general theory established in \cite[\S 2]{G} and \cite[Appendix]{Y simp}, 
if a differentiable averaging operator $S : Y \to \mathcal M_+$ exists, then for any $\alpha \in H_n(X)$, 
\begin{equation} \label{ineq:mass}
\|\alpha\|_1 \le n!\, (\mathcal L [S])^n\, \mass (\alpha)
\end{equation}
holds.
Here, the value $\mathcal L [S]$ %%from above which 
is defined as follows.
For $y \in Y$, 
\[
\mathcal L S_y = \limsup_{z \to y} \frac{\|S(z)-S(y)\|}{d(z,y)} \text{ and }
\mathcal L [S] = \sup_{y \in Y} \frac{\mathcal L S_y}{\|S(y)\|}.
\]

We recall a concrete construction of a differentiable averaging operator.
For $R > 0$ and $y \in Y$, we set $S_R(y) \in \mathcal M_+$ as 
\[
S_R(y) = 1_{B(y,R)} \cdot m_Y.
\]
Here, $1_A$ is the characterizing function of $A \subset Y$.
For every $\epsilon > 0$, we define $S_{R,\epsilon} : Y \to \mathcal M_+$ by
\[
S_{R, \epsilon} (y) = \frac{1}{\epsilon} \int_{R-\epsilon}^R S_{R'}(y) \, d R'.
\]
Its norm is $\|S_{R,\epsilon}(y)\| = \frac{1}{\epsilon} \int_{R-\epsilon}^R v_y(R') \, d R'$ and is not less than $v_y(R-\epsilon)$.
Here, $v_z(r) = m_Y(B(z,r))$ for $z \in Y$ and $r > 0$.
Given a Lipschitz function $\psi = \psi_{R,\epsilon} : [0,\infty) \to [0,1]$ defined as 
\[
\psi (t) = \left\{ 
\begin{array}{ll}
1 & \text{if } t \le R - \epsilon \\
(R - t) / \epsilon & \text{if } t \in [R- \epsilon, R] \\
0 & \text{if }  t \ge  R,
\end{array}
\right.
\]
we can write $S_{R,\epsilon}(y) = \psi(d(y,\cdot))\, m_Y$ for any $y \in Y$.
% \[
% S_{R,\epsilon}(y)(A) = \int_A \psi(d(y,z)) \, d m_Y(z)
% \]
% for any Borel set $A \subset Y$.

We can check $S_{R,\epsilon}$ is a differentiable averaging operator as follows.
Since $m_Y$ is $\pi_1(X)$-invariant, the maps $S_R$ and $S_{R,\epsilon}$ are $\pi_1(X)$-equivariant.
Since $m$ is absolutely continuous in $\mathcal H_X^n$, so is $m_Y$ in $\mathcal H_Y^n$.
One can check that $S_{R,\epsilon}$ is differentiable at $m_Y$-almost everywhere with respect to the differentiable structure of $Y$, where, the differentiable structure on Alexandrov spaces are defined by Otsu and Shioya \cite{OS}.
Indeed, the differential $D_y S_{R,\epsilon}(\gamma^+(0))$ %%, which is a signed measure, 
of $S_{R,\epsilon}$ at $y$ along a geodesic $\gamma$ starting from $y = \gamma(0)$ is calculated as 
\begin{align*}
\left( D_y S_{R,\epsilon} (\gamma^+(0))  \right)(A) 
%%&= \int_A \psi'(d(y,z)) (-\cos \angle(\Uparrow_y^z, \gamma^+(0))) \, d m_Y(z) \\
&= \frac{1}{\epsilon} \int_{A \cap A(y; R-\epsilon, R)} \cos \angle (\Uparrow_y^z, \gamma^+(0)) \,d m_Y(z)
\end{align*}
for any Borel set $A \subset Y$, 
where $A(z;r,r')$ is the annulus around $z \in Y$ of radii between $r$ and $r'$ for $r \le r'$.
%%By \cite{OS}...

To estimate $\mathcal L[S_{R,\epsilon}]$, we use the Bishop-Gromov volume growth estimate as follows. 
We obtain
% \begin{align*}
% \| S_{R,\epsilon}(z) - S_{R,\epsilon}(y) \| 
% &= \frac{1}{\epsilon} \int_{R-\epsilon}^R m_Y(B(z,R') \triangle B(y,R')) \, d R' \\
% &\le \frac{1}{\epsilon} \int_{R-\epsilon}^R 
%%m_Y(B(y, R' + |y z|) - B(y, R')) 
% v_y(R' + |y z|) - v_y(R')
% \, d R' .
% \end{align*}
\[
\mathcal L(S_{R,\epsilon})_y = \sup_{\xi \in \Sigma_y} \|D_y S_{R,\epsilon}(\xi)\| \le 
\frac{m_Y(A(y; R-\epsilon, R))}{\epsilon}
\]
%%Dividing this inequality by $d(z,y)$ and taking limit-sup as tending $d(z,y) \to 0$, we have 
%%%Here, $B \triangle B' = (B - B') \cup (B' - B)$ for $B, B' \subset Y$.
It follows 
% \[
% \mathcal L (S_{R,\epsilon})_y 
%\le \frac{1}{\epsilon} \int_{R-\epsilon}^R s_y(R') \, d R' .
%= \frac{1}{\epsilon} ( v_{y}(R) - v_{y}(R-\epsilon) )
%%%\le \limsup_{t \to 0} \frac{1}{\epsilon} \int_{R-\epsilon}^R \frac{v_y(R'+t) - v_y(R')}{t} \, d R'
% \le \frac{1}{\epsilon} \int_{R-\epsilon}^R s_y(R') \, d R' 
% = \frac{1}{\epsilon} (v_y(R) - v_y(R-\epsilon)).
% \]
from %%%the Bishop-Gromov volume growth estimate 
$\bg(K,N)$, 
\begin{align*}
\frac{\mathcal L(S_{R,\epsilon})_y}{\|S_{R,\epsilon}(y)\|} 
% &\le \frac{s_y(R)}{v_y(R)}
% = \lim_{t \to 0} \int_R^{R+t} \frac{v_y(R'+t) - v_y(R')}{v_y(R)} \, d R' \\ 
% &\le \lim_{t \to 0} \int_R^{R+t} \frac{v_{K,N}(R'+t) - v_{K,N}(R)}{v_{K,N}(R)}
\le \frac{v_y(R) - v_y(R-\epsilon)}{\epsilon \cdot v_y(R-\epsilon)} 
\le C_{K,N}({R,\epsilon}).
\end{align*}
Here, setting $\bar v(R') = \bar v_{K, N}(R') = \int_0^{R'} \sn_{K/(N-1)}^{N-1}(t) \, d t$, 
\[
C_{K,N}({R,\epsilon}) := \frac{\bar v (R) - \bar v(R-\epsilon)}{\epsilon \cdot \bar v(R-\epsilon)}.
\]

% Therefore, for any $\alpha \in H_n(X)$,
% \[
% \|\alpha \|_1 \le n!\, C_{R,\epsilon}^n\, \mass(\alpha)
% \]
% holds.
Since %%$\lim_{R \to \infty} \lim_{\epsilon \to 0} C_{R, \epsilon} = \sqrt{-N K}$ and 
$\mass([X]) = \mathcal H^n(X)$ (\cite[Theorem 0.1]{Y simp}), by using \eqref{ineq:mass} and by tending $\epsilon \to 0$ and $R \to \infty$, we obtain 
\[
\|X \| \le n! \sqrt{-K (N-1)}^{\,n} \mathcal H^n(X).
\]
It completes the proof of Theorem \ref{thm:simp vol}.
\end{proof}

\begin{remark} \upshape
Due to Petrunin \cite{Pt ALVS} and Zhang and Zhu \cite{ZZ}, it is known that for $n$-dimensional Alexandrov space $X$ of curvature $\ge \kappa$, the metric measure space $(X,\mathcal H^n)$ satisfies the curvature-dimension condition $\cd((n-1)\kappa, n)$.
Therefore, Corollary \ref{cor:simp vol} is implied by Theorem \ref{thm:simp vol} via \cite{Pt ALVS} and \cite{ZZ}.

If there exists a compact orientable $n$-dimensional Alexandrov space $X$ without boundary of curvature $\ge \kappa$ with $\kappa < 0$ which has nonnegative Ricci curvature with respect to some reference measure $m$ so that $m \ll \mathcal H^n$ and $\supp(m) = X$, then Theorem \ref{thm:simp vol} yields $\|X\| = 0$.
\end{remark}

\noindent
{\bf Acknowledgments}.
The first author is supported by Research Fellowships of the Japan
Society for the Promotion of Science for Young Scientists.

\end{document}